\documentclass[11pt,a4paper,leqno]{amsart}
\usepackage[utf8]{inputenc}
\usepackage{version}
\usepackage{amsmath,amssymb,amsfonts, amsthm, yfonts}
\usepackage{calrsfs}
\usepackage{esint}
\usepackage{hyperref}
\usepackage[parfill]{parskip}	
\usepackage{mathtools}
\usepackage{bbm}
\usepackage{enumitem}

\usepackage{amsmath,amsthm,amssymb,latexsym,epsfig,graphicx,subfigure}
\newcommand*\mathinhead[2]{\texorpdfstring{$#1$}{#2}}

\hypersetup{
	colorlinks=true,
	linkcolor=magenta,
	filecolor=black,      
	urlcolor=blue,
	citecolor=cyan,
}

\numberwithin{equation}{section}

\theoremstyle{plain}
\newtheorem*{thm*}{Theorem}
\newtheorem{thm}{Theorem}[section]
\newtheorem{lem}[thm]{Lemma}
\newtheorem{prop}[thm]{Proposition}
\newtheorem{cor}[thm]{Corollary}
\theoremstyle{definition}
\newtheorem{defn}{Definition}[section]

\theoremstyle{remark}
\newtheorem*{rem}{\textbf{Remark}}

\DeclareMathAlphabet{\pazocal}{OMS}{zplm}{m}{n}

\begin{document}
	\title[$(1/2,+)$-caloric capacity of Cantor sets]{On the $(1/2,+)$-caloric capacity of Cantor sets.}
	\author{Joan Hernández}\thanks{The author was supported by the grant
		PRE2021-098469 (Mineco, Spain).}
	\maketitle
	
	\begin{abstract}
		In the present paper we characterize the $(1/2,+)$-caloric capacity (associated with the $1/2$-fractional heat equation) of the usual corner-like Cantor set of $\mathbb{R}^{n+1}$. The results obtained for the latter are analogous to those found for Newtonian capacity. Moreover, we also characterize the BMO and $\text{Lip}_\alpha$ variants ($0<\alpha<1$) of the $1/2$-caloric capacity in terms of the Hausdorff contents $\pazocal{H}^n_\infty$ and $\pazocal{H}^{n+\alpha}_\infty$ respectively.
		
		\bigskip
		
		
		\medskip
		
	\end{abstract}

	\section{Introduction}
	\label{sec1}
	
	The work done by Mateu, Prat \& Tolsa in \cite{MPrTo} and subsequently by Mateu \& Prat in \cite{MPr}, motivates the study of caloric capacities in a similar manner as it has been done for classical objects such as analytic, harmonic or Newtonian capacities. In the former article, $(1,1/2)$-Lipschitz caloric capacity was introduced and, along with it, the notion of equivalence between the nullity of this quantity and the removability of compact subsets for the heat equation, i.e. the one associated with the differential operator $\Theta:=(-\Delta_x)+\partial_t$, where $(x,t)\in\mathbb{R}^n\times \mathbb{R}$. As one may suspect, the different nature of the \textit{spatial} and \textit{temporal} variables of the equation is key when formalizing the previous notions. Indeed, for a given compact subset $E\subset \mathbb{R}^{n+1}$ and $f$ a solution of the heat equation in $\mathbb{R}^{n+1}\setminus{E}$, that is $\Theta f = 0$ in $\mathbb{R}^{n+1}\setminus{E}$; the characterization of the removability of $E$ (i.e. the possibility to extend the solution to the whole $\mathbb{R}^{n+1}$) in terms of its caloric capacity, can be done as in \cite{MPrTo} if $f$ satisfies the following normalization conditions: 
	\begin{equation*}
		\|\nabla_x f\|_{L^{\infty}(\mathbb{R}^{n+1})}<\infty, \hspace{1cm} \|\partial_t^{1/2} f\|_{\ast, p}<\infty,
	\end{equation*}
	where the norm $\|\cdot\|_{\ast, p}$ stands for the usual BMO norm of $\mathbb{R}^{n+1}$ but computed with respect to \textit{parabolic} cubes. The choice to endow $\mathbb{R}^{n+1}$ with a parabolic metric topology, that is, the metric induced by the distance
    \begin{equation*}
        d_p\big( (x,t),(y,s) \big):=\max\{|x-y|,|t-s|^{1/2}\};
    \end{equation*}
    becomes natural in light of the results presented by Hofmann \cite[Lemma 1]{Ho} or Hofmann \& Lewis \cite[Theorem 7.4]{HoL}. Indeed, recall that a function $f(x,t)$ defined in $\mathbb{R}^{n+1}$ is Lip $1/2$ (or Hölder $1/2$) in the $t$ variable if
    \begin{equation*}
        \|f\|_{\text{Lip}_{1/2},t}:=\sup_{\substack{x\in \mathbb{R}^n\\ t,u\in \mathbb{R}, t\neq u}} \frac{|f(x,t)-f(x,u)|}{|t-u|^{1/2}}<\infty.
    \end{equation*}
    Then, the following estimate holds
	\begin{equation*}
		\|f\|_{\text{Lip}_{1/2},t}\lesssim \|\nabla_x f\|_{L^\infty(\mathbb{R}^{n+1})}+\|\partial_t^{1/2}f\|_{\ast,p}.
	\end{equation*}
	This implies that if $f$ satisfies the previously mentioned normalization conditions, it also satisfies a \textit{ $($1,1/2$\,)$-Lipschitz condition}, that in turn explains the name given to the caloric capacity presented in \cite{MPrTo}.\medskip\\
	In the footsteps of \cite{MPrTo}, Mateu \& Prat studied the corresponding caloric capacities associated with the fractional heat equation in \cite{MPr}. That is, the equation associated with the pseudo-differential operator $\Theta^s:=(-\Delta_x)^s+\partial_t$, for $0<s<1$. In such paper, the authors carried out their analysis distinguishing the cases $s=1/2$, $1/2<s<1$ and $0<s<1/2$, focusing mainly on the first. The study of the second and third cases turned out to be rather technical and cumbersome, and it resulted in the possibility to only obtain the value of the critical dimension of the capacity for the second case, and the bound from above by a certain $s$-parabolic Hausdorff content for the third. Nevertheless, the study of the first case was quite fruitful, deducing a similar removability result for the $\Theta^{1/2}$-equation, as it was done for the genuine heat equation. For instance, if $f$ is a solution of the $1/2$-heat equation in $\mathbb{R}^{n+1}\setminus{E}$ satisfying
	\begin{equation*}
		\|f\|_{L^\infty(\mathbb{R}^{n+1})}<\infty,
	\end{equation*}
	then $E$ will be removable if and only if the $1/2$-caloric capacity of $E$ is null. Notice that a particularity of the choice $s=1/2$ is that the spatial and temporal variables need no longer to be distinguished in the normalization condition. In addition, the proper topology to endow $\mathbb{R}^{n+1}$ also becomes the usual euclidean topology. All in all, the study of this case is simplified and, in fact, additional results can be obtained in the planar setting $(n=1)$, such as the \textit{non-comparability} of the $1/2$-caloric capacity with analytic and Newtonian capacities, despite that the three share critical dimension 1.\medskip\\
	In the present paper we aim at obtaining further results for the case $s=1/2$ and give a more precise description of the $1/2$-caloric capacity, once its definition is restricted to positive Borel regular measures. Such version will be referred to as ($1/2,+$)-caloric capacity. For instance, Section \ref{sec4} is devoted to the estimation of the ($1/2,+$)-caloric capacity of the usual corner-like Cantor set of $\mathbb{R}^{n+1}$, that finally yields Corollary \ref{cor4.3}. The behavior obtained is similar to that described by Eiderman in \cite{Ei} for radial nonnegative kernels. In our setting, the kernel will be nonnegative, but not radial. Nevertheless, to circumvent such inconvenience, the author compares the usual ($1/2,+$)-caloric capacity with an auxiliary one, defined also through a nonnegative kernel but lacking a temporal indicator function. Such feature turns out to be essential to deduce, in a rather straightforward manner, the desired estimate. In fact, it also yields the comparability of the ($1/2,+$)-caloric capacity with the analogous capacity associated with the \textit{conjugate} operator (and thus conjugate equation),
	\begin{equation*}
		\overline{\Theta}^{1/2}:=(-\Delta_x)^{1/2}-\partial_t.
	\end{equation*}
	In Sections \ref{sec5} and \ref{sec6} we characterize the BMO and $\text{Lip}_\alpha$ variants ($0<\alpha<1$) of the $1/2$-caloric capacity, obtaining in Theorems \ref{thm5.3} and \ref{thm6.4} their comparability with $\pazocal{H}^n_\infty$ and $\pazocal{H}^{n+\alpha}_\infty$ respectively. The previous study has been motivated by the one carried out for the BMO variant of analytic capacity by Astala, Iwaniec \& Martin in \cite[\textsection 13.5.1]{AIMar}, which in turn was inspired by \cite{Ka}; and that for the $\text{Lip}_\alpha$ variant of the same capacity presented by Mel'nikov \cite{Me}.\newpage
	\textit{About the notation used in this text}: as usual, the letter $C$ stands for an absolute positive constant that can depend on the dimension of the ambient space, and whose value may change at different occurrences. The notation $A\lesssim B$ means that there exists a positive absolute (dimensional) constant, such that $A\leq CB$. Moreover, $A\approx B$ is equivalent to $A\lesssim B \lesssim A$. Also, $A \simeq B$ will mean $A= CB$. We also emphasize that the gradient symbol $\nabla$ will refer to $(\nabla_x,\partial_t)$, with $x\in \mathbb{R}^n$ and $t\in \mathbb{R}$. 
	
	\bigskip
	\section{Notation and preliminary results}
	\label{sec2}
	\bigskip
	
	Our ambient space will be $\mathbb{R}^{n+1}$, and a generic point will be denoted as $\overline{x}=(x,t)\in\mathbb{R}^{n+1}$, where $x\in\mathbb{R}^n$ will be usually referred to as the \textit{spatial} variable, and $t\in \mathbb{R}$ the \textit{time} variable. Let $\Theta^{1/2}$ be the \textit{$1/2$-heat operator}, that is
	\begin{equation*}
		\Theta^{1/2} := (-\Delta)^{1/2} + \partial_t,
	\end{equation*}
	where $(-\Delta)^{1/2}=(-\Delta_x)^{1/2}$ is a pseudo-differential operator known as the \textit{$1/2$-Laplacian} with respect to the spatial variable. It may be defined through its Fourier transform,
	\begin{equation*}
		\widehat{(-\Delta)^{1/2}}f(\xi,t)=|\xi|\widehat{f}(\xi,t),
	\end{equation*}
	or by its integral representation
	\begin{align*}
		(-\Delta)^{1/2} f(x,t)&\simeq \text{p.v.}\int_{\mathbb{R}^n}\frac{f(x,t)-f(y,t)}{|x-y|^{n+1}}\text{d}\pazocal{L}^{n}(y)  \\
		&\simeq \int_{\mathbb{R}^n}\frac{f(x+y,t)-2f(x,t)+f(x-y,t)}{|y|^{n+1}}\text{d}\pazocal{L}^n(y). 
	\end{align*}
	The reader may find more details about the properties of such operator in \cite[\textsection{3}]{DPV} or \cite{St}. Borrowing the notation of \cite{MPr}, let $P$ be the fundamental solution of the $1/2$-heat equation in $\mathbb{R}^{n+1}$, which is given by \cite[Eq. 2.2]{Va}
	\begin{equation*}
		P(\overline{x})=\frac{t}{\big( t^2+|x|^2 \big)^{(n+1)/2}}\, \chi_{\{t>0\}}(\overline{x}),
	\end{equation*}
	where $\chi$ is the usual indicator function. For short, such expression will be simply written as
	\begin{equation*}
		P(\overline{x})=\frac{t}{|\overline{x}|^{n+1}}\, \chi_{\{t>0\}}.
	\end{equation*}
	Notice that the previous kernel is not differentiable at any point of the form $(x,0)$. Another fundamental function that will appear in the sequel is $P^{\ast}$,
	\begin{equation*}
		P^{\ast}(\overline{x}):=P(-\overline{x})= \frac{-t}{|\overline{x}|^{n+1}}\, \chi_{\{t<0\}}.
	\end{equation*}
	Observe that, on the one hand,
	\begin{align*}
		(-\Delta)^{1/2}&P^{\ast}(\overline{x})=C'\int_{\mathbb{R}^n}\frac{P^{\ast}(x+y,t)-2P^{\ast}(x,t)+P^{\ast}(x-y,t)}{|y|^{n+1}}\text{d}\pazocal{L}^n(y)\\
		&\hspace{-0.2cm}=C'\int_{\mathbb{R}^n}\frac{P(-x-y,-t)-2P(-x,-t)+P(-x+y,-t)}{|y|^{n+1}}\text{d}\pazocal{L}^n(y) =(-\Delta)^{1/2}P(-\overline{x}),
	\end{align*}
	while on the other,
	\begin{equation*}
		\partial_tP^{\ast}(\overline{x})=-\partial_t P(-\overline{x}).
	\end{equation*}
	Therefore, if we define the operator
	\begin{equation*}
		\overline{\Theta}^{1/2} := (-\Delta)^{1/2} - \partial_t,
	\end{equation*}
	we have that
	\begin{equation*}
		\overline{\Theta}^{1/2}P^{\ast}(\overline{x})= \Theta^{1/2}P(-\overline{x}),
	\end{equation*}
	implying that $P^{\ast}$ is the fundamental solution of $\overline{\Theta}^{1/2}$.
	
	\bigskip
	
	\begin{defn}[$1/2$-caloric capacity]
		\label{def2.1}
		Given a compact subset $E\subset \mathbb{R}^{n+1}$, define its \textit{$1/2$-caloric capacity} as
		\begin{equation*}
			\gamma_{\Theta^{1/2}}(E)=\sup |\langle T, 1 \rangle| ,
		\end{equation*}
		where the supremum is taken over all distributions $T$ with $\text{supp}(T)\subseteq E$ satisfying
		\begin{equation*}
			\|P\ast T\|_{\infty}:=\|P\ast T\|_{L^\infty(\mathbb{R}^{n+1})} \leq 1.
		\end{equation*}
		Such distributions will be called \textit{admissible for $\gamma_{\Theta^{1/2}}(E)$}.
	\end{defn}

    We also define the ($1/2,+$)-\textit{caloric capacity}, denoted by $\gamma_{\Theta^{1/2},+}$, in the same way as $\gamma_{\Theta^{1/2}}$, but with the supremum only taken over positive Borel regular measures $\mu$ with $\text{supp}(\mu)\subseteq E$ and such that $\|P\ast \mu\|_{\infty} \leq 1$.\medskip\\
    We shall also introduce yet another variant of the previous capacity, $\widetilde{\gamma}_{\Theta^{1/2},+}$, that will be referred to as ($1/2,+$)\textit{-symmetric caloric capacity}. Admissible measures for $\widetilde{\gamma}_{\Theta^{1/2},+}$ must also satisfy $\|P^{\ast}\ast \mu\|_{\infty} \leq 1$ as well as an $n$-growth condition with constant 1. Recall that a Borel measure $\mu$ in $\mathbb{R}^{n+1}$ has $s$-\textit{growth \text{\normalfont{(}}with constant $C$\text{\normalfont{)}}} if there exists some absolute constant $C>0$ such that
	\begin{equation*}
		\mu\big( B(\overline{x},r) \big) \leq Cr^s, \hspace{0.5cm} \text{for all }\, \overline{x}\in \mathbb{R}^{n+1},\, r>0.
	\end{equation*}
	It is clear that this property is invariant if formulated using cubes instead of balls. In the sequel we shall present a result that justifies the name given to $\widetilde{\gamma}_{\Theta^{1/2},+}$. Also, it is straightforward that 
	\begin{equation*}
		\widetilde{\gamma}_{\Theta^{1/2},+}(E)\leq \gamma_{\Theta^{1/2},+}(E)\leq \gamma_{\Theta^{1/2}}(E).
	\end{equation*}
	We have analogous definitions associated with the operator $\overline{\Theta}^{1/2}$, giving rise to the objects $\gamma_{\overline{\Theta}^{1/2}}, \gamma_{\overline{\Theta}^{1/2},+}$ and $\widetilde{\gamma}_{\overline{\Theta}^{1/2},+}$. Moreover, it is usual to extend all of the above definitions to a greater variety of sets. Namely, if $E\subseteq \mathbb{R}^{n+1}$ is any Borel set,
	\begin{equation*}
		\gamma_{\Theta^{1/2}}(E):=\sup_{\substack{K\subseteq E\\ K \text{compact}}} \gamma_{\Theta^{1/2}}(K),
	\end{equation*}
	and similarly for the rest of capacities. \medskip\\
	The main properties of $\gamma_{\Theta^{1/2}}$ regarding localization and comparability to the Hausdorff measure are exhaustively covered in \cite{MPr}. We highlight \cite[Theorem 4.2]{MPr}, that concerns the equivalence between the null sets for the $\gamma_{\Theta^{1/2}}$ capacity and the removable sets for the $\Theta^{1/2}$-equation. More precisely:
	
	\begin{defn}[$1/2$-caloric removability]
		\label{def2.2}
		A compact subset $E\subset \mathbb{R}^{n+1}$ is said to be \textit{$1/2$-caloric removable} if any bounded function $f:\mathbb{R}^{n+1}\to \mathbb{R}$ satisfying the $1/2$-heat equation in $\mathbb{R}^{n+1}\setminus{E}$, also satisfies the same equation in $E$.
	\end{defn}
	
	\begin{thm}{\normalfont{(\cite[Theorem 4.2]{MPr})}\textbf{.}}
		\label{thm2.1}
		A compact subset $E\subset \mathbb{R}^{n+1}$ is $1/2$-caloric removable if and only if $\gamma_{\Theta^{1/2}}(E)=0$.
	\end{thm}
	
	Moreover, the $1/2$-caloric capacity of a subset is tightly related to a certain Hausdorff content of the latter. Such relation is conditioned, in particular, by a growth restriction that admissible distributions for $\gamma_{\Theta^{1/2}}$ must satisfy.
	
	\begin{thm}
		\label{thm2.2}
		Let $T$ be a distribution in $\mathbb{R}^{n+1}$ with $\|P\ast T\|_\infty\leq 1$. If $\varphi$ is a $\pazocal{C}^1$ function supported on $Q\subset \mathbb{R}^{n+1}$ with $\|\nabla \varphi\|_\infty \leq \ell(Q)^{-1}$, then
		\begin{equation*}
			|\langle T, \varphi \rangle |\leq C \ell(Q)^n,
		\end{equation*}
		for some absolute dimensional constant $C>0$.
	\end{thm}
	
	If the previous property holds for a distribution $T$, we say that $T$ has $n$-\textit{growth} (\textit{with constant $C$}). It can be checked that this definition agrees with the usual definition of $n$-growth if $T$ is a Borel measure.
	
	\begin{thm}
		\label{thm2.3}
		For every compact subset $E\subset \mathbb{R}^{n+1}$,
		\begin{equation*}
			\gamma_{\Theta^{1/2}}(E)\lesssim \pazocal{H}^{n}_{\infty}(E).
		\end{equation*}
		Moreover, if $\text{dim}_{\pazocal{H}}(E)>n$, then $\gamma_{\Theta^{1/2}}(E)>0$.
	\end{thm}
	
	From the previous results we infer that the only possible candidate for the critical dimension of $\gamma_{\Theta^{1/2}}$ is $n$. And, in fact, this is the case, since there exist examples of compact subsets with positive $n$-dimensional Hausdorff measure, and some of them are $1/2$-caloric removable \cite[\textsection{5}]{MPr}, while other are not (they can be obtained as a consequence of \cite[Theorem 4.3]{MPr} via subsets of graphs of Lipschitz functions with positive $\pazocal{H}^n$-measure).\medskip\\
	One may also ask if for the planar case ($n=1$), the capacity $\gamma_{\Theta^{1/2}}$ is comparable to analytic capacity or Newtonian capacity, two classical objects related to complex analysis and potential theory sharing the same critical dimension with $\gamma_{\Theta^{1/2}}$. By \cite[Proposition 6.1]{MPr} we see that this is not the case, since there exist subsets of $\mathbb{R}^2$ with positive $\gamma_{\Theta^{1/2}}$ and null Newtonian capacity (horizontal line segments), and null $\gamma_{\Theta^{1/2}}$ but positive analytic capacity (vertical line segments).
	
	\bigskip
	\section{Properties of \mathinhead{\gamma_{\Theta^{1/2},+}}{}}
	\label{sec3}
	\bigskip
	
	In the sequel we will be concerned with estimating the ($1/2,+$)-caloric capacity of a family of generalized Cantor sets of $\mathbb{R}^{n+1}$. Previous to that, we shall present some important features of $\gamma_{\Theta^{1/2},+}$ that we consider of their own interest.
	
	\begin{prop} 
		\label{prop3.1}
		Let $E\subset \mathbb{R}^{n+1}$ be a Borel subset and $\lambda>0, \tau\in \mathbb{R}^{n+1}$. Set $\tau(E):= E+\tau$ and denote by $\lambda E$ the dilation of $E$ by $\lambda$. The following identities hold:
		\begin{enumerate}
			\itemsep-0.25em
			\item[\textit{1}.] Translation invariance: $\gamma_{\Theta^{1/2},+}(E)=\gamma_{\Theta^{1/2},+}(\tau(E))$.
			\item[\textit{2}.] $\gamma_{\Theta^{1/2},+}(\lambda E)=\lambda^n\gamma_{\Theta^{1/2},+}(E)$.
			\item[\textit{3}.] Outer regularity: If $(E_k)_k$ is a nested sequence of compact subsets of $\mathbb{R}^{n+1}$ that decreases to $\pazocal{E}:=\cap_{k=1}^\infty E_k$,
			\begin{equation*}
				\lim_{k\to\infty}\gamma_{\Theta^{1/2},+}(E_k)=\gamma_{\Theta^{1/2},+}(\pazocal{E}).
			\end{equation*}
			\item[\textit{4}.] Countable subadditivity: Let $E_1,E_2,\ldots$ be disjoint Borel subsets of $\mathbb{R}^{n+1}$. Then,
			\begin{equation*}
				\gamma_{\Theta^{1/2},+}\Bigg( \bigcup_{j=1}^\infty E_j \Bigg)\leq \sum_{j=1}^{\infty} \gamma_{\Theta^{1/2},+}(E_j).
			\end{equation*}
		\end{enumerate}
		
	\end{prop}
	\begin{proof}
		During this proof we shall write $\gamma_{+}:=\gamma_{\Theta^{1/2},+}$ to ease notation. To verify \textit{1}, we pick $E \subset \mathbb{R}^{n+1}$ compact and prove that for any $\mu$ admissible for $\gamma_{+}(E)$ there exists a measure $\mu_{\tau}$, admissible for $\gamma_{+}(\tau(E))$, such that $\mu(E)=\mu_{\tau}(\tau(E))$. It is clear that once this property is verified, the result will follow.
        Let $\mu$ be admissible for $\gamma_{+}(E)$ and define $\mu_{\tau}(X):=\mu(X-\tau)$, for any $X\subseteq \mathbb{R}^{n+1}$ that is $\mu$-measurable. This way $\mu_{\tau}$ is clearly a positive Borel regular measure supported on $\tau(E)$ with $\mu_{\tau}(\tau(E))=\mu(E)$. In addition, for any $\overline{x}\in\mathbb{R}^{n+1}$,
		\begin{align*}
			|P\ast \mu_{\tau}(\overline{x})|=\bigg\rvert\int_{\tau(E)} P(\overline{x}-\overline{y})\text{d}\mu_{\tau}(\overline{y})\bigg\rvert=\bigg\rvert\int_{E} P(\overline{x}-\tau-\overline{u})\text{d}\mu(\overline{u})\bigg\rvert=|P\ast\mu(\overline{x}-\tau)|\leq 1,
		\end{align*}
		implying that $\mu_{\tau}$ is admissible for $\gamma_{+}(\tau(E))$ and we are done. To deal with $E$ an arbitrary Borel subset of $\mathbb{R}^{n+1}$, just notice that by Theorem \ref{thm2.2} admissible measures for $\gamma_+$ are locally finite and Borel regular, and thus Radon \cite[Corollary 1.11]{Ma}. So the quantity $\mu(E)$ can be computed as the limit $\lim_{k\to\infty}\mu(E_k)$, where $E_k$ is a proper sequence of compact subsets that approximates $E$.\medskip\\
		The proof of \textit{2} is analogous. Indeed, take the measure $\mu_\lambda(X):=\lambda^{n}\mu(\lambda^{-1}X)$, supported on $\lambda E$, and just notice that for any $\overline{x}\in \mathbb{R}^{n+1}$,
		\begin{align*}
			|P\ast \mu_{\lambda}(\overline{x})|=\bigg\rvert\int_{\lambda E} P(\overline{x}-\overline{y})\text{d}\mu_{\lambda}(\overline{y})\bigg\rvert=\lambda^n\bigg\rvert\int_{E} P(\overline{x}-\lambda\overline{u})\text{d}\mu(\overline{u})\bigg\rvert=|P\ast\mu(\lambda^{-1}\overline{x})|\leq 1.
		\end{align*}
		Moving on to \textit{3}, observe that $\gamma_+(\pazocal{E})\leq \gamma_+(E_k)$ for any $k$. Hence $\gamma_+(\pazocal{E})\leq \lim_{k\to\infty}\gamma_{+}(E_k)$ and we are left to prove the converse inequality. To do so, let us pick for each $k$ an admissible measure $\mu_k$ for $\gamma_+(E_k)$ with
		\begin{equation*}
			\gamma_+(E_k)-\frac{1}{k} \leq \mu_k(E_k) \leq \gamma_+(E_k),
		\end{equation*}
		We shall verify that there exists an admissible measure $\mu$ for $\gamma_+(\pazocal{E})$ so that for each test function $\varphi$, $\lim_{k\to\infty} \langle \mu_k, \varphi \rangle = \langle \mu, \varphi \rangle$, where we have used the notation $\langle \mu, \varphi \rangle := \int \varphi \, \text{d}\mu$. If this is the case, for $\varphi$ test function with $\varphi \equiv 1$ in a neighborhood of $\pazocal{E}$,
		\begin{equation*}
			\lim_{k\to\infty} \gamma_+(E_k)\leq \lim_{k\to\infty}  \mu_k(E_k) = \lim_{k\to\infty} \langle \mu_k, \varphi \rangle = \langle \mu, \varphi \rangle = \mu(\pazocal{E}) \leq \gamma_+(\pazocal{E}),
		\end{equation*}
		and we would be done. To construct such $\mu$, let $\varphi\in \pazocal{C}_c^{\infty}(\mathbb{R}^{n+1})$ and observe that $\langle \mu_k, \varphi \rangle = \langle \Theta^{1/2}P\ast \mu_k, \varphi \rangle = \langle P\ast \mu_k, \overline{\Theta}^{1/2}\varphi \rangle$. By assumption $P\ast \mu_k$ belongs to the unit ball of $L^{\infty}(\mathbb{R}^{n+1})\cong L^1(\mathbb{R}^{n+1})^\ast$ and moreover, proceeding as in \cite[\textsection 3]{MPr}, it is clear that $\overline{\Theta}^{1/2}\varphi \in L^1(\mathbb{R}^{n+1})$. Therefore, since $L^1(\mathbb{R}^{n+1})$ is separable, by the sequential version of Banach-Alaoglu's theorem we may assume that there exists some $S\in L^{\infty}(\mathbb{R}^{n+1})$ with $\|S\|_\infty\leq 1$ and $P\ast \mu_k\to S$ as $k\to \infty$ in a $\text{weak}^\star$-$L^\infty$ sense. Therefore,
		\begin{equation*}
			\lim_{k\to\infty} \langle \mu_k, \varphi \rangle = \langle S, \overline{\Theta}^{1/2}\varphi \rangle, \hspace{0.5cm} \forall \varphi \in \pazocal{C}^\infty_c(\mathbb{R}^{n+1}).
		\end{equation*}
		Let us define a distribution (a priori) $\mu$ acting on test functions as $\langle \mu, \varphi \rangle := \langle S, \overline{\Theta}^{1/2}\varphi \rangle$, so that we have $\lim_{k\to\infty} \langle \mu_k, \varphi \rangle = \langle \mu, \varphi \rangle$ for any $\varphi \in \pazocal{C}_c^\infty(\mathbb{R}^{n+1})$. Observe that by the latter identity, for any $\varphi\geq 0$ we have $\langle \mu, \varphi \rangle \geq 0$. It is not difficult to prove that such property implies that $\mu$ is a distribution of order $0$ (we refer the reader to the proof of \cite[Theorem 2.7]{Ca}, for example), so applying \cite[Theorem 2.5]{Ca} and Riesz's representation theorem, we deduce that in fact $\mu$ is a positive Radon measure. In addition, since the supports of $\mu_k$ are contained in $E_k$ and $E_k\downarrow \pazocal{E}$, it follows that $\text{supp}(\mu)\subseteq \pazocal{E}$. Therefore, if we prove that $\|P\ast \mu\|_\infty\leq 1$ we will be done, since $\mu$ would become an admissible measure for $\gamma_+(\pazocal{E})$. \medskip\\
		Such estimate will follow from the equality $P\ast \mu = S$. To verify it, we regularize $P\ast \mu_k$ and $\mu_k$: take $\psi\in \pazocal{C}_c^\infty(B(0,1))$ positive and radial with $\int \psi = 1$ and set $\psi_\varepsilon:=\varepsilon^{-(n+1)}\psi(\cdot/\varepsilon)$. Then,
		\begin{equation*}
			\lim_{k\to\infty} \big( \psi_\varepsilon \ast P \ast \mu_k \big)(\overline{x})= \psi_\varepsilon \ast S (\overline{x}), \hspace{0.5cm} \overline{x}\in \mathbb{R}^{n+1},
		\end{equation*}
		since $P\ast \mu_k$ converges to $S$ in a $\text{weak}^\star$-$L^\infty$ sense. On the other hand, as $\psi_\varepsilon\ast P \in \pazocal{C}^\infty(\mathbb{R}^{n+1})$ and by definition $\mu_k$ converges to $\mu$ in the weak topology of distributions, we have
		\begin{equation*}
			\lim_{k\to\infty} \big( \psi_\varepsilon \ast P \ast \mu_k \big)(\overline{x})= \big( \psi_\varepsilon \ast P \ast \mu \big) (\overline{x}), \hspace{0.5cm} \overline{x}\in \mathbb{R}^{n+1}.
		\end{equation*}
		So $\psi_\varepsilon\ast S = \psi_\varepsilon \ast P \ast \mu$ for every $\varepsilon>0$, so $S=P\ast \mu$, and the proof of \textit{3} is complete.\medskip\\
		Finally we prove \textit{4}. Abusing notation, let us set $E:= \bigcup_{j=1}^\infty E_j$, which is also a Borel subset of $\mathbb{R}^{n+1}$, and fix $K\subset E \subset \mathbb{R}^{n+1}$ compact. Let $\mu$ be admissible for $\gamma_{+}(K)$. Observe that for any $X\subseteq \mathbb{R}^{n+1}$ $\mu$-measurable, one has
		\begin{align*}
			\mu(X)=\mu\Bigg( \bigcup_{j=1}^\infty (K\cap E_j)\cap X \Bigg)=\sum_{j=1}^\infty \mu|_{K\cap E_j}(X),
		\end{align*}
		so in particular, since $K$ is also a Borel set and thus $\mu$-measurable,
		\begin{align*}
			\mu(K)=\sum_{j=1}^\infty \mu|_{E_j}(K).
		\end{align*}
		If we take the supremum over all admissible measures for $\gamma_{+}(K)$ on both sides of the previous inequality, we have
		\begin{align*}
			\gamma_{+}(K)\leq \sum_{j=1}^\infty \;  \sup_{\substack{\text{supp}(\mu)\subseteq K \\ \|P\ast \mu\|_\infty \leq 1}}\; \mu|_{E_j}(K).
		\end{align*}
		We claim that for each $j\geq 1$ the following is true:
		\begin{equation}
			\label{eq3.1}
			\sup_{\substack{\text{supp}(\mu)\subseteq K \\ \|P\ast \mu\|_\infty \leq 1}}\; \mu|_{E_j}(K) \leq \gamma_{+}(E_j\cap K).
		\end{equation}
		To verify such estimate we assume that it does not hold and reach a contradiction. So suppose that there exists $\mu$ admissible for $\gamma_{+}(K)$ with
		\begin{equation*}
			\mu|_{E_j}(K)>\gamma_{+}(E_j\cap K).
		\end{equation*}
		Then, for any compact subset $F\subseteq E_j\cap K $ we have $\mu|_{E_j}(K) > \gamma_{+}(F)$. Clearly $\mu|_{F}$ is admissible for $\gamma_+(F)$. Indeed, for any $\overline{x}\in \mathbb{R}^{n+1}$,
		\begin{equation*}
			|P\ast \mu|_{F}(\overline{x})| = \int_{F} P(\overline{x}-\overline{y})\text{d}\mu(\overline{y})\leq \int_{K} P(\overline{x}-\overline{y})\text{d}\mu(\overline{y})\leq \|P\ast \mu\|_{\infty} \leq 1,
		\end{equation*}
		and the Borel regularity follows from that of $\mu$ and \cite[Theorem 1.9]{Ma}, that can be applied by the $n$-growth of $\mu$. Thus $\gamma_{+}(F)\geq \mu(F)$. Hence, by hypothesis, 
		\begin{equation*}
			\mu(E_j\cap K)>\mu(F), \hspace{0.75cm} \forall F \subseteq E_{j}\cap K \;\text{ with } F \text{ compact},
		\end{equation*}
		which contradicts that $\mu$ enjoys an inner regularity property, since it is a Radon measure on $\mathbb{R}^{n+1}$. Therefore, \eqref{eq3.1} must hold, which implies
		\begin{equation*}
			\gamma_{+}(K)\leq \sum_{j=1}^\infty\gamma_{+}(E_j\cap K)\leq \sum_{j=1}^\infty\gamma_{+}(E_j).
		\end{equation*}
		Then, since $K$ was any compact subset contained in $E$, the desired estimate follows.
	\end{proof}

    \begin{rem}
		The argument we have presented for property \textit{3} can be easily adapted for general distributions. That is, it can be checked that, in fact, $\gamma_{\Theta^{1/2}}$ also enjoys the outer regularity property.
	\end{rem}

    The next result describes the behavior of $\gamma_{\Theta^{1/2},+}$ under canonical reflections of $\mathbb{R}^{n+1}$.
	
	\begin{prop}
		\label{prop3.2}
		Let $E\subset \mathbb{R}^{n+1}$ be a Borel set and for each $i\in\{1,\ldots,n\}$ denote by $\pazocal{R}_{i}$ the reflection with respect to the hyperplane $\{x_i=0\}$, and  by $\pazocal{R}_{t}$ the reflection with respect to $\{t=0\}$. Then,
		\begin{equation*}
			\gamma_{\Theta^{1/2},+}(E)= \gamma_{\Theta^{1/2},+}(\pazocal{R}_{i}(E)), \hspace{0.5cm} 1\leq i\leq n,
		\end{equation*}
		and moreover,
		\begin{equation*}
			\gamma_{\Theta^{1/2},+}(E)= \gamma_{\overline{\Theta}^{1/2},+}(\pazocal{R}_{t}(E)).
		\end{equation*}
	\end{prop}
	
	\begin{proof}
		Fix $i\in\{1,\ldots,n\}$ and check, as in the proof of properties \textit{1} and \textit{2} of Proposition \ref{prop3.1}, that for any $\mu$ admissible for and $\gamma_{\Theta^{1/2},+}(E)$, there is $\mu_{i}$, admissible for $\gamma_{\Theta^{1/2},+}(\pazocal{R}_{i}(E))$, such that $\mu(E)=\mu_{i}(\pazocal{R}_{i}(E))$. So we fix $\mu$ admissible for $\gamma_{\Theta^{1/2},+}(E)$ and define
		\begin{equation*}
			\mu_{i}(X):=\mu\big(\pazocal{R}_{i}^{-1}(X)\big), \hspace{0.5cm} \forall X\subseteq \mathbb{R}^{n+1}\; \mu\text{-measurable}.
		\end{equation*}
		Again, $\mu_i$ is a positive Borel regular measure supported on $\pazocal{R}_{i}(E)$ such that $\mu(E)=\mu_i(\pazocal{R}_{i}(E))$. Finally, to verify the admissibility of $\mu_i$, notice that for any $\overline{x}\in \mathbb{R}^{n+1}$ we have 
		\begin{align*}
			|P\ast \mu_{i}(\overline{x})|&= \int_{\pazocal{R}_{i}(E)} P(\overline{x}-\overline{y})\text{d}\mu_{i}(\overline{y}) =\ \int_{E} P\big(\overline{x}-\pazocal{R}_{i}(\overline{u})\big)\text{d}\mu(\overline{u})
		\end{align*}
		Observe that $\pazocal{R}_{i}(\overline{u})=(u_1,\ldots,-u_i,\ldots,u_{n+1})$, so using the particular definition of $P$,
		\begin{align*}
			\int_{E} P\big(\overline{x}-\pazocal{R}_{i}(\overline{u})\big)\text{d}\mu(\overline{u}) = \int_{E} P\big(\pazocal{R}_{i}(\overline{x})-\overline{u}\big)\text{d}\mu(\overline{u}) =  P\ast \mu(\pazocal{R}_{i}(\overline{x}))  
			\leq 1,
		\end{align*}
		that is what we wanted to prove. On the other hand, if $i=n+1$, that is, if $x_i=t$, the computations are similar, but let us make them explicit to emphasize the role of the indicator function, that is responsible for the change of $\Theta^{1/2}$ into $\overline{\Theta}^{1/2}$:
		\begin{align*}
			|P^{\ast}\ast \mu_{n+1}(\overline{x})|&= \int_{E} P^{\ast}\big(\overline{x}-\pazocal{R}_{t}(\overline{u})\big)\text{d}\mu(\overline{u}) =  \int_E \frac{-t-u}{|\pazocal{R}_{t}(\overline{x})-\overline{u}|^{n+1}}\chi_{\{-t-u>0\}}(\overline{u})\text{d}\mu(\overline{u}) \\
			&= \int_{E}P\big( \pazocal{R}_{t}(\overline{x})-\overline{u}\big)\text{d}\mu(\overline{u}) =P\ast\mu(\pazocal{R}_{t}(\overline{x}))\leq 1,
		\end{align*}
		that implies the desired result.
	\end{proof}
	
	\begin{rem}
		Observe that combining the first result of Proposition \ref{prop3.1} and Proposition \ref{prop3.2}, the latter result also holds for any affine canonical reflection. That is, any reflection with respect to hyperplanes of the form $\{x_i=c\}$ or $\{t=c\}$, for any $c\in\mathbb{R}$. This implies, in particular, that if $E$ presents any temporal axis of symmetry, then its $\gamma_{\Theta^{1/2},+}$ and $\gamma_{\overline{\Theta}^{1/2},+}$ capacities coincide.
	\end{rem}
	\begin{rem}
		Notice that we have also obtained that if the measure $\mu$ satisfies the condition $\|P\ast \mu\|_{\infty}\leq 1$, then
		\begin{equation*}
			\|P\ast \mu_{\tau}\|_{\infty}\leq 1 \hspace{0.5cm} \text{and} \hspace{0.5cm}  \|P\ast \mu_{i}\|_{\infty}\leq 1, \hspace{0.25cm} \text{for} \hspace{0.25cm} i=1,2,\ldots,n.
		\end{equation*}
	\end{rem}
	
	\bigskip
	\subsection{Comparability between \mathinhead{\gamma_{\Theta^{1/2},+}}{} and \mathinhead{\widetilde{\gamma}_{\Theta^{1/2},+}}{}}
	\label{subsec3.2}
	\bigskip
	
	One of the main characteristics of the kernels $P$ and $P^{\ast}$ is the presence of an indicator function with respect to the $t$-variable. Such fact seems to endow the temporal axis with a distinct feature when it comes to constructing removable sets for the $\Theta^{1/2}$-equation, as it is exemplified in \cite[Proposition 6.1]{MPr} with the vertical line segment $\{0\}\times [0,1]$. And what about the time-reflected line segment $\{0\}\times [0,-1]$? It is clear, by the translation invariance of $\gamma_{\Theta^{1/2},+}$, that its capacity is equally $0$.\medskip\\
	When trying to find a subset $E\subset \mathbb{R}^{n+1}$ with non-comparable $\gamma_{\Theta^{1/2},+}$ and $\gamma_{\overline{\Theta}^{1/2},+}$ capacities, the above trivial observation suggests that it may be not possible. In fact, the following result was a first motivation to carry out the study of the present subsection:
	
	\begin{prop}
		\label{prop3.3}
		The $\gamma_{\Theta^{1/2},+}$ capacity of any non-horizontal line segment is null.
	\end{prop}
	
	\begin{proof}
		It is clear that we may assume $n=1$, that is, the ambient space is $\mathbb{R}^{2}$. Denote by $E$ the unit segment with one of its end-points at the origin and with angle $\alpha\in (0,\pi)$ between the positive direction of the $x$-axis and $E$. We shall follow the same method of proof given for \cite[Proposition 6.1] {MPr}, that is: we will assume $\gamma_{\Theta^{1/2},+}(E)>0$ and reach a contradiction.\medskip\\
		Under the previous assumption there exists an admissible measure $\mu$ for $\gamma_{\Theta^{1/2},+}(E)$ with $\mu(E)>0$. Let us parameterize $E$ as $u\mapsto (u\cos{\alpha}, u\sin{\alpha}), \, u\in[0,1]$ and note that since $\mu$ has linear growth (Theorem \ref{thm2.2}), given $\eta>0$ we can take $c\in(0,1)$ such that
		\begin{equation*}
			\mu\big(\big\{  (u\cos{\alpha}, u\sin{\alpha})\;:\; c\leq u \leq 1 \big\}\big) < \eta.
		\end{equation*}
		Writing explicitly the normalization condition $\|P\ast \mu\|_{\infty}\leq 1$, we have
		\begin{equation*}
			P\ast \mu(\overline{x})=\int_0^1\frac{t-u\sin{\alpha}}{\big( x-u\cos{\alpha} \big)^2+\big( t-u\sin{\alpha} \big)^2}\,\chi_{\{t-u\sin{\alpha}>0\}}\text{d}\mu(u)\leq 1,
		\end{equation*}
		for $\pazocal{L}^2$-a.e. $\overline{x}\in \mathbb{R}^{2}$. Therefore, if we set $F:=\{  (u\cos{\alpha}, u\sin{\alpha})\;:\; 0\leq u <c\, \}$ and choose $\overline{x}=(u_0\cos{\alpha}, u_0\sin{\alpha})\in F$, we get
		\begin{equation*}
			\sin{\alpha}\int_0^{u_0}\frac{\text{d}\mu(u)}{u_0-u}\leq 1.
		\end{equation*}
		So for any $\overline{x}\in F$ there exists $\ell=\ell(\overline{x})>0$ such that
		\begin{equation*}
			\sin{\alpha}\int_{u_0-\ell}^{u_0}\frac{\text{d}\mu(u)}{u_0-u}\leq \eta.
		\end{equation*}
		Hence, since $F$ is an interval, there exists a finite number of almost disjoint intervals $I_j$ with $|I_j|=\ell_j=\ell(\overline{x}_j)$ such that $F\subset \bigcup_{j=1}^N I_j$ and
		\begin{equation*}
			\mu(I_j)=\int_{u_{0,j}-\ell_j}^{u_{0,j}}\text{d}\mu(u)\leq \int_{u_{0,j}-\ell_j}^{u_{0,j}}\frac{\ell_j}{u_{0,j}-u}\text{d}\mu(u)\leq \ell_j\frac{\eta}{\sin{\alpha}}.
		\end{equation*}
		All in all,
		\begin{align*}
			\mu(E) < \mu(F)+\eta \lesssim \sum_{j=1}^{N}\mu(I_j)+\eta \leq \eta\Bigg(\frac{1}{\sin{\alpha}}\sum_{j=1}^N \ell_n + 1  \Bigg) \lesssim \eta\bigg( \frac{c}{\sin{\alpha}} + 1 \bigg),
		\end{align*}
		and this leads to a contradiction, since $\eta$ can be chosen arbitrarily small. Therefore, $\gamma_{\Theta^{1/2},+}(E)=0$ and this, together with the first and fourth properties of Proposition \ref{prop3.1} suffices to generalize the result for any other line segment.
	\end{proof}
	
	\begin{rem}
		We have proved this result only for $\gamma_{\Theta^{1/2},+}$ just for the sake of simplicity and to focus our study on its properties. However, by exactly the same method of proof of \cite[Proposition 6.1]{MPr} (involving the approximation of distributions by signed measures) one can obtain the same result for $\gamma_{\Theta^{1/2}}$.
	\end{rem}
	
	The second aspect that motivated the study of the comparability between $\gamma_{\Theta^{1/2},+}$ and $\widetilde{\gamma}_{\Theta^{1/2},+}$ is related to the different equivalent definitions admitted by the latter capacity. We stress that in the proofs given for the forthcoming results, we will \textit{exploit the fact that $P$ is a nonnegative kernel}. To ease notation, let us simply set
	\begin{equation*}
		\widetilde{\gamma}_+:=\widetilde{\gamma}_{\Theta^{1/2},+}.
	\end{equation*}
	As it is pointed out in \cite[\textsection 4]{MPr}, one of the main advantages of working with $\widetilde{\gamma}_{+}$ instead of just $\gamma_{\Theta^{1/2},+}$ is that it can be characterized in a similar manner as those capacities defined through anti-symmetric kernels, by means of the $L^2$-bound of a particular operator. To make such property explicit, we shall first introduce some notation. \medskip\\
    For a given real compactly supported Borel regular measure $\mu$ with $n$-growth, we define the operator $\pazocal{P}_\mu$ acting on elements of $L^1_{\text{loc}}(\mu)$ as
	\begin{equation*}
		\pazocal{P}_{\mu}f(\overline{x}):=\int_{\mathbb{R}^{n+1}}P(\overline{x}-\overline{y})f(\overline{y})\text{d}\mu(\overline{y}), \hspace{0.5cm} \overline{x}\notin \text{supp}(\mu).
	\end{equation*}
	Since $P(\overline{x})\lesssim |\overline{x}|^{-n}$, it is clear that the previous expression is defined pointwise on $\mathbb{R}^{n+1}\setminus{\text{supp}(\overline{x})}$; but the convergence of the integral may fail for $\overline{x}\in \text{supp}(\mu)$. This motivates the definition of a truncated version of $\pazocal{P}$,
	\begin{equation*}
		\pazocal{P}_{\mu,\varepsilon}f(\overline{x}):=\int_{|\overline{x}-\overline{y}|>\varepsilon}P(\overline{x}-\overline{y})f(\overline{y})\text{d}\mu(\overline{y}),\hspace{0.5cm} \overline{x}\in \mathbb{R}^{n+1}, \; \varepsilon>0.
	\end{equation*}
	For a given $1\leq p \leq \infty$, we will say that $\pazocal{P}_\mu f$ belongs to $L^p(\mu)$ if the $L^p(\mu)$-norm of the truncations $\|\pazocal{P}_{\mu,\varepsilon}f\|_{L^p(\mu)}$ is  uniformly bounded on $\varepsilon$, and we write
	\begin{equation*}
		\|\pazocal{P}_\mu f\|_{L^p(\mu)}:=\sup_{\varepsilon>0}\|\pazocal{P}_{\mu,\varepsilon}f\|_{L^p(\mu)}
	\end{equation*}
	We will say that the operator $\pazocal{P}_\mu$ is bounded on $L^p(\mu)$ if the operators $\pazocal{P}_{\mu,\varepsilon}$ are bounded on $L^p(\mu)$ uniformly on $\varepsilon$, and we equally set
	\begin{equation*}
		\|\pazocal{P}_\mu\|_{L^p(\mu)\to L^p(\mu)}:=\sup_{\varepsilon>0}\|\pazocal{P}_{\mu,\varepsilon}\|_{L^p(\mu)\to L^p(\mu)}.
	\end{equation*}
	We also introduce the \textit{transposed} operator associated with $\pazocal{P}_\mu$,
	\begin{equation*}
		\pazocal{P}_{\mu}^{\ast}f(\overline{x}):=\int_{\mathbb{R}^{n+1}}P^{\ast}(\overline{x}-\overline{y})f(\overline{y})\text{d}\mu(\overline{y}), \hspace{0.5cm} f\in L^1_{\text{loc}}(\mu),\;\overline{x}\notin \text{supp}(\mu);
	\end{equation*}
	with all its corresponding definitions relative to truncations and $L^p(\mu)$-boundedness. Now we are ready to state a crucial property of the capacity $\widetilde{\gamma}_+$. To compactify notation, let $\Sigma(E)$ be the collection of all positive Borel regular measures supported on $E$ that have $n$-growth with constant $1$.
 
	\begin{thm}{\normalfont{(\cite[Theorem 4.3]{MPr})}\textbf{.}}
		\label{thm3.4}
		For any $E\subset \mathbb{R}^{n+1}$ compact subset,
		\begin{equation*}
			\widetilde{\gamma}_{+}(E)\approx \gamma_{2,+}(E) := \sup\Big\{\mu(E)\;:\; \mu\in \Sigma(E),\; \|\pazocal{P}_\mu\|_{L^2(\mu)\to L^2(\mu)}\leq 1  \Big\},
		\end{equation*}
		where the implicit constant in the above estimate does not depend on $E$.
	\end{thm}
	
	We move on with our study introducing, for any $\overline{x}\neq 0$, the following kernel
	\begin{equation*}
		P_{\text{\normalfont{sy}}}(\overline{x}):=\frac{1}{2}\big[ P(\overline{x})+P(-\overline{x}) \big] = \frac{1}{2}\bigg[ \frac{t}{|\overline{x}|^{n+1}}\chi_{\{t>0\}}-\frac{t}{|\overline{x}|^{n+1}}\chi_{\{t<0\}} \bigg] = \frac{|t|}{2|\overline{x}|^{n+1}}.
	\end{equation*}
	
	\begin{lem}
		\label{lem3.5}
		Let $E\subset \mathbb{R}^{n+1}$ be compact and define
		\begin{equation*}
			\gamma_{\text{\normalfont{sy}},+}(E):=\sup_{\mu}\Big\{ \mu(E)\;:\; \mu\in \Sigma(E),\; \|P_{\text{\normalfont{sy}}}\ast \mu\|_{\infty}\leq 1  \Big\}.
		\end{equation*}
		Then,
		\begin{equation*}
			\frac{1}{2} \gamma_{\text{\normalfont{sy}},+}(E) \leq \widetilde{\gamma}_{+}(E) \leq \gamma_{\text{\normalfont{sy}},+}(E).
		\end{equation*}
	\end{lem}
	\begin{proof}
		Take $\mu$ any admissible measure for $\widetilde{\gamma}_{+}(E)$ and observe that by definition of $P_{\text{\normalfont{sy}}}$,
		\begin{equation*}
			\|P_{\text{\normalfont{sy}}}\ast \mu\|_{\infty}\leq \frac{1}{2}\Big[ \|P\ast \mu\|_{\infty}+\|P^{\ast}\ast \mu\|_{\infty} \Big]\leq 1,
		\end{equation*}
		that yields $\widetilde{\gamma}_{+}(E) \leq \gamma_{\text{\normalfont{sy}},+}(E)$. Conversely, if we consider any $\mu \in \Sigma(E)$ with $\|P_{\text{\normalfont{sy}}}\ast \mu\|_{\infty}\leq 1$, since $P$ is nonnegative, $P\leq 2P_{\text{\normalfont{sy}}}$ and $P^{\ast}\leq 2P_{\text{\normalfont{sy}}}$, and therefore
		\begin{equation*}
			\|P\ast \mu\|_{\infty}\leq 2, \hspace{0.5cm} \|P^{\ast}\ast \mu\|_{\infty}\leq 2.
		\end{equation*}
		So $\mu/2$ becomes admissible for $\widetilde{\gamma}_{+}(E)$ and we deduce the remaining inequality.
	\end{proof}
	
	Hence $\widetilde{\gamma}_{+}$ is comparable to the capacity defined through the symmetric kernel $P_{\text{\normalfont{sy}}}$. That is the reason to call it $(1/2,+)$-\textit{symmetric} caloric capacity. Let us now proceed with introducing another auxiliary capacity:
	\begin{equation*}
		\widetilde{\gamma}_{+}'(E):= \sup_{\mu}\Big\{\mu(E)\;:\; \mu\in \Sigma(E),\; \|P\ast \mu\|_{L^\infty(\mu)}\leq 1, \; \|P^{\ast}\ast \mu\|_{L^\infty(\mu)}\leq 1  \Big\}.
	\end{equation*}
	
	\begin{lem}
		\label{lem3.6}
		For a compact set $E\subset \mathbb{R}^{n+1}$,
		\begin{equation*}
			\widetilde{\gamma}_{+}(E)\lesssim \widetilde{\gamma}_{+}'(E).
		\end{equation*}
	\end{lem}
	\begin{proof}
		As it is pointed out at the beginning of the proof of \cite[Theorem 4.3]{MPr}, the arguments to show the above estimate are standard. However, for the sake of completeness, we shall present them (they are inspired by \cite[Lemma 5.4]{MaP}). \medskip\\
		Let $\mu\in \Sigma(E)$ with $ \|P\ast \mu\|_{\infty}\leq 1$ and $ \|P^{\ast}\ast \mu\|_{\infty}\leq 1$. We aim to prove $ \|P\ast \mu\|_{L^\infty(\mu)}\lesssim 1$ and $ \|P^{\ast}\ast \mu\|_{L^\infty(\mu)}\lesssim 1$. We will study the first inequality (the computations for the second will be analogous), and to verify it we will argue by contradiction. That is, we assume that there exists $F\subseteq E$ with $\mu(F)>0$ so that $P\ast \mu$ attains arbitrarily big values there (notice that this does not contradict $ \|P\ast \mu\|_{\infty}\leq 1$, since  $F$ may be $\pazocal{L}^{n+1}$-null). Under this assumption, consider $0<\varepsilon<1$ and $\overline{x}\in F$, and observe that
		\begin{align*}
			\|P\ast \mu|_{B(\overline{x},\varepsilon)}\|_{L^1(B(\overline{x},\varepsilon/2))}&\lesssim  \int_{B(\overline{x},\varepsilon/2)}\bigg(\int_{B(\overline{x},\varepsilon)}\frac{\text{d}\mu(\overline{y})}{|\overline{z}-\overline{y}|^{n}}\bigg)\text{d}\pazocal{L}^{n+1}(\overline{z})\\
			&\leq  \int_{B(\overline{x},\varepsilon)}\bigg(\int_{B(\overline{y},2\varepsilon)} \frac{\text{d}\pazocal{L}^{n+1}(\overline{z})}{|\overline{z}-\overline{y}|^{n}} \bigg)\text{d}\mu(\overline{y}) \lesssim \varepsilon\,\mu(B(\overline{x},\varepsilon))\\
			&\leq \varepsilon^{n+1} \simeq \pazocal{L}^{n+1}\big( B(\overline{x},\varepsilon/2) \big).
		\end{align*}
		So by the Lebesgue differentiation theorem we can pick $\overline{z}\in B(\overline{x},\varepsilon/2)$ with $|P\ast \mu(\overline{z})|\leq \|P\ast \mu\|_\infty\leq 1$ and such that for a positive dimensional constant $A_1$,
		\begin{equation*}
			\big\rvert P\ast \mu|_{B(\overline{x},\varepsilon)}(\overline{z})\big\rvert \leq A_1.
		\end{equation*}
		Notice that the constant $A_1$ does not depend on $\overline{x}$, since the procedure above can be repeated for any other point obtaining the same estimate. Therefore,
		\begin{align*}
			\big\rvert P\ast \mu|_{\mathbb{R}^{n+1}\setminus{B(\overline{x},\varepsilon)}}(\overline{x})&-P\ast \mu(\overline{z})\big\rvert \\
			&=\Bigg\rvert \int_{|\overline{x}-\overline{y}|>\varepsilon}P(\overline{x}-\overline{y})\text{d}\mu(\overline{y})- \int_{\mathbb{R}^{n+1}}P(\overline{z}-\overline{y})\text{d}\mu(\overline{y}) \Bigg\rvert\\
			&\leq \int_{|\overline{x}-\overline{y}|>\varepsilon}\big\rvert P(\overline{x}-\overline{y})-P(\overline{z}-\overline{y}) \big\rvert \text{d}\mu(\overline{y})+A_1.
		\end{align*}
		
		Write $\overline{x}=(x,t_x),\,\overline{y}=(y,t_y)$ and $\overline{z}=(z,t_z)$ and consider the auxiliary point $\widehat{x}=(z-y,t_x-t_y)$. Applying the mean value theorem (component-wise) exactly as it is done in the proof of \cite[Lemma 2.1]{MPr} we have
		\begin{align*}
			\int_{|\overline{x}-\overline{y}|>\varepsilon}\big\rvert &P(\overline{x}-\overline{y})-P(\overline{z}-\overline{y}) \big\rvert \text{d}\mu(\overline{y}) \\
			&\leq  \int_{|\overline{x}-\overline{y}|>\varepsilon}\big\rvert P(\overline{x}-\overline{y})-P(\widehat{x}) \big\rvert \text{d}\mu(\overline{y}) + \int_{|\overline{x}-\overline{y}|>\varepsilon}\big\rvert P(\widehat{x})-P(\overline{z}-\overline{y}) \big\rvert \text{d}\mu(\overline{y})\\
			&\leq A_2\varepsilon \int_{|\overline{x}-\overline{y}|>\varepsilon} \frac{\text{d}\mu(\overline{y})}{|\overline{x}-\overline{y}|^{n+1}}+A_3\varepsilon \int_{|\overline{x}-\overline{y}|>\varepsilon} \frac{\text{d}\mu(\overline{y})}{|\overline{x}-\overline{y}|^{n+1}}
		\end{align*}
		
		Splitting the domain of integration into the annuli $A_j:=B\big(\overline{x},2^{j+1}\varepsilon\big)\setminus{B\big(\overline{x},2^{j}\varepsilon\big)}$ for $j\geq 0$, and using the $n$-growth of $\mu$ we obtain, regarding the above integrals,
		\begin{align*}
			\varepsilon \int_{|\overline{x}-\overline{y}|>\varepsilon}\frac{\text{d}\mu(\overline{y})}{|\overline{x}-\overline{y}|^{n+1}}=\varepsilon\sum_{j=1}^\infty\int_{A_j}\frac{\text{d}\mu(\overline{y})}{|\overline{x}-\overline{y}|^{n+1}} \leq \varepsilon\sum_{j=1}^\infty\frac{\big(2^{j+1}\varepsilon\big)^{n}}{(2^j\varepsilon)^{n+1}} = \sum_{j=1}^\infty\frac{1}{ 2^j} = 1.
		\end{align*}
		Therefore,
		\begin{align*}
			\big\rvert P\ast \mu|_{\mathbb{R}^{n+1}\setminus{B(\overline{x},\varepsilon)}}(\overline{x}) \big\rvert &\leq |P\ast \mu|_{\mathbb{R}^{n+1}\setminus{B(\overline{x},\varepsilon)}}(\overline{x})- P\ast \mu (\overline{z})|+|P\ast \mu(\overline{z})|\\
			&\leq A_2+A_3+A_1+\|P\ast \mu \|_{\infty}<\infty,
		\end{align*}
		and as $\overline{x}\in F$ was arbitrary and the previous constants are absolute, we get the contradiction we were looking for.
	\end{proof}
	
	Observe that in the previous proof we have deduced that for a given $\mu\in \Sigma(E)$,
	\begin{equation}
		\label{eq3.2}
		\text{if } \hspace{0.25cm} \|P\ast \mu\|_{\infty}\leq 1, \hspace{0.25cm} \text{ then } \hspace{0.25cm} \|P\ast \mu\|_{L^\infty(\mu)}\lesssim 1.
	\end{equation}
	
	To establish one additional inequality, we continue by introducing some more standard terminology adapted to our particular setting.
	
	\begin{defn}[$m$-dimensional Calderón-Zygmund kernel]
		\label{def3.1}
		Consider the (kernel) function $K:\mathbb{R}^{n+1}\times \mathbb{R}^{n+1}\setminus{\{(\overline{x},\overline{y})\in \mathbb{R}^{n+1}\times\mathbb{R}^{n+1}: \overline{x}=\overline{y}\}}\to \mathbb{C}$, with the property that there exist $C_1>0$ and $m>0$ such that for any $\overline{x}\neq \overline{y}$,
		\begin{equation*}
			|K(\overline{x},\overline{y})|\leq \frac{C_1}{|\overline{x}-\overline{y}|^m}.
		\end{equation*}
		Suppose also that for any $\overline{x}'$ with $|\overline{x}-\overline{x}'|\leq |\overline{x}-\overline{y}|/2$, there exist $C_2>0,\eta>0$ so that
		\begin{equation*}
			|K(\overline{x},\overline{y})-K(\overline{x}',\overline{y})|+|K(\overline{y},\overline{x})-K(\overline{y},\overline{x}')|\leq \frac{C_2|\overline{x}-\overline{x}'|^{\eta}}{|\overline{x}-\overline{y}|^{m+\eta}}.
		\end{equation*}
		If such estimates are satisfied, $K$ is called an $m$-\textit{dimensional Calderón-Zygmund kernel}. It is clear that the function $K(\overline{x},\overline{y}):=P(\overline{x}-\overline{y})$ is an $n$-dimensional Calderón-Zygmund kernel with $\eta=1$.
	\end{defn}
	
	We are now ready to state \cite[Theorem 3.21]{To4} in our particular context. Let us also remark that such result (in a general setting of $n$-dimensional Calderón-Zygmund operators with respect to measures with $n$-growth which may be non-doubling) was originally proved in \cite{NTV}, although we present a refinement that can be obtained from the results in \cite{NTV2} or \cite{To2}, as mentioned in \cite[\textsection 3.7.2]{To4}.
	
	\begin{thm}
		\label{thm3.7}
		Let $E\subset \mathbb{R}^{n+1}$ be a compact subset and $\mu\in \Sigma(E)$. Then, the operator $\pazocal{P}_\mu$ extends to a bounded operator on $L^2(\mu)$ if and only if there exists $c>0$ such that
		\begin{equation*}
			\|\pazocal{P}_{\mu,\varepsilon}\chi_Q\|_{L^2(\mu|_Q)}\leq c\mu(2Q)^{1/2}\hspace{0.5cm} \text{and} \hspace{0.5cm} \|\pazocal{P}_{\mu,\varepsilon}^{\ast}\chi_Q\|_{L^2(\mu|_Q)}\leq c\mu(2Q)^{1/2},
		\end{equation*}
		uniformly on $\varepsilon>0$, for any cube $Q\subset \mathbb{R}^{n+1}$.
	\end{thm}
	In light of such result, the following lemma follows from the nonnegativity of $P$:
	\begin{lem}
		\label{lem3.8}
		For $E\subset \mathbb{R}^{n+1}$ compact subset,
		\begin{equation*}
			\widetilde{\gamma}_{+}'(E)\lesssim \gamma_{2,+}(E).
		\end{equation*}
	\end{lem}
	
	\begin{proof}
		Let $\mu$ be admissible for $ \widetilde{\gamma}_{+}'(E)$. Since by definition of this capacity we have $\|\pazocal{P}_{\mu,\varepsilon}1\|_{L^\infty(\mu)}\leq 1$ and $\|\pazocal{P}_{\mu,\varepsilon}^{\ast}1\|_{L^\infty(\mu)}\leq 1$ uniformly on $\varepsilon>0$, and both the kernel $P$ and measure $\mu$ are nonnegative, we have for every $\varepsilon>0$ and every cube $Q\subset \mathbb{R}^{n+1}$,
		\begin{align*}
			\|\pazocal{P}_{\mu,\varepsilon}&\chi_Q\|_{L^2(\mu|_Q)}=\Bigg( \int_Q\bigg\rvert \int_{Q\cap \{|\overline{x}-\overline{y}|>\varepsilon\}} P(\overline{x}-\overline{y})\text{d}\mu(\overline{y}) \bigg\rvert^2\text{d}\mu(\overline{x}) \Bigg)^{1/2}\\
			&\leq \Bigg( \int_Q\bigg\rvert \int_{|\overline{x}-\overline{y}|>\varepsilon} P(\overline{x}-\overline{y})\text{d}\mu(\overline{y}) \bigg\rvert^2\text{d}\mu(\overline{x}) \Bigg)^{1/2} \leq \|\pazocal{P}_{\mu,\varepsilon}1\|_{L^\infty(\mu)}\,\mu(Q)^{1/2}\leq \mu(2Q)^{1/2},
		\end{align*}
		and analogously for $\pazocal{P}_{\mu,\varepsilon}^{\ast}$. Therefore, by a direct application of Theorem \ref{thm3.7} we deduce the desired estimate.
	\end{proof}
	
	Combining all the above results we deduce the following corollary, which encapsulates the different ways to understand $\widetilde{\gamma}_+$:
	
	\begin{cor}
		\label{cor3.9}
		For $E\subset \mathbb{R}^{n+1}$ compact subset,
		\begin{equation*}
			\widetilde{\gamma}_+(E)\approx \widetilde{\gamma}_{+}'(E) \approx \gamma_{2,+}(E) \approx \gamma_{\text{\normalfont{sy}},+}(E).
		\end{equation*}
	\end{cor}
	
	To be able to compare $\gamma_{\Theta^{1/2},+}$ and $\widetilde{\gamma}_+$, we will need two additional definitions, firstly introduced in \cite{To1}. We remark that 
	\begin{itemize}
		\item[] in the forthcoming Definitions \ref{def3.2} and \ref{def3.3}, as well as Lemma \ref{lem3.10}, $\mu$ \textit{will always be a positive compactly supported Borel regular measure on $\mathbb{R}^{n+1}$ with $n$-growth}.
	\end{itemize}
	Notice that the previous conditions ensure that the degree of growth of $\mu$ is the same as the degree of homogeneity of $P$, understood as an $n$-dimensional Calderón-Zygmund kernel on $\mathbb{R}^{n+1}$. Moreover, as we have pointed out in the proof of Proposition \ref{prop3.1}, observe that $\mu$ is locally finite and therefore becomes a Radon measure.
	
	\begin{defn}[$\text{BMO}_\rho(\mu)$]
		\label{def3.2}
		Given $\rho>1$ and $f\in L^1_{\text{loc}}(\mu)$, we say that $f$ belongs to $\textit{BMO}_\rho(\mu)$ if for some constant $c>0$,
		\begin{equation*}
			\sup_{Q} \frac{1}{\mu(\rho Q)} \int_Q \big\rvert f(\overline{x}) - f_{Q,\mu}\big\rvert \text{d}\mu(\overline{x}) \leq c,
		\end{equation*}
		where the supremum is taken among all cubes such that $\mu(Q)\neq 0$, and $f_{Q,\mu}$ is the average of $f$ in $Q$ with respect to $\mu$. The infimum over all values $c$ satisfying the above inequality is the so-called $\textit{BMO}_\rho(\mu)$ \textit{norm of} $f$.
	\end{defn}
	Notice that if $f\in L^\infty(\mu)$, then $f\in \text{BMO}_\rho(\mu)$ for any $\rho>1$. Moreover, for any $a\in \mathbb{R}$ and any cube $Q\subset \mathbb{R}^{n+1}$,
	\begin{align*}
		\int_Q \big\rvert f(\overline{x}) &-f_{Q,\mu}\big\rvert \text{d}\mu(\overline{x}) \\
		&\leq \Bigg[  \int_Q \big\rvert f(\overline{x}) - a \big\rvert \text{d}\mu(\overline{x}) + \int_Q \big\rvert a - f_{Q,\mu}\big\rvert \text{d}\mu(\overline{x}) \Bigg] \leq 2 \int_Q \big\rvert f(\overline{x}) - a\big\rvert \text{d}\mu(\overline{x}).
	\end{align*}
	Therefore, if for each cube $Q$ with $\mu(Q)\neq 0$ we are able to find $c_Q$ so that
	\begin{equation*}
		\frac{1}{\mu(\rho Q)}\int_Q \big\rvert f(\overline{x}) - c_Q\big\rvert \text{d}\mu(\overline{x}) \leq c,
	\end{equation*}
	where $c$ is constant independent of $Q$, we deduce that $f$ belongs to $\text{BMO}_\rho(\mu)$.
	
	\begin{defn}[$\mu$-weakly bounded]
		\label{def3.3}
		We will say that the operator $\pazocal{P}_\mu$ is \textit{$\mu$-weakly bounded} if for any cube $Q\subset \mathbb{R}^{n+1}$,
		\begin{align*}
			\big\rvert \big\langle \pazocal{P}_{\mu,\varepsilon} \chi_{Q}, \chi_{Q} \big\rangle \big\rvert := \Bigg\rvert \int_Q \bigg( \int_{Q\cap \{|\overline{x}-\overline{y}|>\varepsilon\}} P(\overline{x}-\overline{y})\text{d}\mu(\overline{y})\bigg)\text{d}\mu(\overline{x}) \Bigg\rvert \leq c \mu(2Q),
		\end{align*}
		uniformly on $\varepsilon>0$.
	\end{defn}
	
	\begin{lem}
		\label{lem3.10}
		Assume that $\|P\ast \mu\|_{L^{\infty}(\mu)}\leq 1$. Then, the operator $\pazocal{P}_\mu$ is $\mu$-weakly bounded and $P^{\ast}\ast \mu \in \text{\normalfont{BMO}}_\rho(\mu)$ for $\rho\geq 2$.
	\end{lem}
	\begin{proof}
		Condition $\|P\ast \mu\|_{L^{\infty}(\mu)}\leq 1$ can be simply rewritten as $\sup_{\varepsilon>0}\|\pazocal{P}_{\mu,\varepsilon}1\|_{L^{\infty}(\mu)}\leq 1$; so for any $\varepsilon>0$, by the nonnegativity of $P$ and $\mu$,
		\begin{align*}
			\Bigg\rvert \int_Q \bigg( \int_{Q\cap \{|\overline{x}-\overline{y}|>\varepsilon\}} &P(\overline{x}-\overline{y})\text{d}\mu(\overline{y})\bigg)\text{d}\mu(\overline{x})  \Bigg\rvert 
			\leq \|\pazocal{P}_{\mu,\varepsilon}1\|_{L^{\infty}(\mu)}\mu(Q) \leq \mu(2Q).
		\end{align*}
		Hence $\pazocal{P}_\mu$ is $\mu$-weakly bounded. Finally, notice that by Tonelli's theorem,
		\begin{equation*}
			\int_{\mathbb{R}^{n+1}} P^{\ast}\ast \mu(\overline{x}) \text{d}\mu(\overline{x}) = \int_{\mathbb{R}^{n+1}} P\ast \mu(\overline{y}) \text{d}\mu(\overline{y}) \leq \mu(\mathbb{R}^{n+1})<\infty,
		\end{equation*}
		so $P^{\ast}\ast \mu \in L^1(\mu) \subset L^1_{\text{loc}}(\mu)$, and $P^{\ast}\ast \mu$ is indeed a candidate to belong to $\text{BMO}_\rho(\mu)$. To estimate its $\text{BMO}_\rho(\mu)$ norm, fix a cube $Q=Q(\overline{x}_0,\ell(Q)), \,\mu(Q)\neq 0$ and consider the characteristic function $\chi_{2Q}$ associated with $2Q$. Denote also $\chi_{2Q^c}:=1-\chi_{2Q}$. Consider the constant
		\begin{equation*}
			c_Q:=P^{\ast}\ast (\chi_{2Q^c}\mu)(\overline{x}_0),  
		\end{equation*}
		that is an expression pointwise well-defined, since $\overline{x}_0 \notin \text{supp}(\chi_{2Q^c}\mu)$. Indeed,
		\begin{align*}
			c_Q = \int_{\mathbb{R}^{n+1}\setminus{2Q}}P^{\ast}(\overline{x}_0-\overline{z})\text{d}\mu(\overline{z})&\leq \int_{\mathbb{R}^{n+1}\setminus{2Q}}\frac{\text{d}\mu(\overline{z})}{|\overline{x}_0-\overline{z}|^{n}}\leq \frac{\mu(\mathbb{R}^{n+1})}{(2\ell(Q))^{n}}<\infty.
		\end{align*}
		Observe that the following estimate holds
		\begin{align*}
			\frac{1}{\mu(\rho Q)}&\int_Q |P^{\ast}\ast \mu(\overline{y})-c_Q|\text{d}\mu(\overline{y})\\
			&\leq \frac{1}{\mu(\rho Q)}\int_Q P^{\ast}\ast (\chi_{2Q}\mu)(\overline{y}) \text{d}\mu(\overline{y})\\
			&\hspace{1.4cm}+\frac{1}{\mu(\rho Q)}\int_Q|P^{\ast}\ast (\chi_{2Q^c}\mu)(\overline{y})-P^{\ast}\ast (\chi_{2Q^c}\mu)(\overline{x}_0)|\text{d}\mu(\overline{y})\\
			&\leq \frac{1}{\mu(\rho Q)}\int_Q\bigg(\int_{2Q}P^{\ast}(\overline{y}-\overline{z}) \text{d}\mu(\overline{z})\bigg)\text{d}\mu(\overline{y})\\
			&\hspace{1.4cm}+\frac{1}{\mu(\rho Q)}\int_Q\bigg( \int_{\mathbb{R}^{n+1}\setminus{2Q}}|P^{\ast}(\overline{y}-\overline{z})-P^{\ast}(\overline{x}_0-\overline{z})| \text{d}\mu(\overline{z})\bigg) \text{d}\mu(\overline{y}) =:I_1+I_2.
		\end{align*}
		For $I_1$, by Tonelli's theorem together with the assumptions $\|P\ast \mu\|_{L^\infty(\mu)}\leq 1$ and $\rho\geq 2$,
		\begin{align*}
			I_1 = \frac{1}{\mu(\rho Q)}\int_{2Q} P\ast \mu (\overline{z})\text{d}\mu(\overline{z}) \leq \frac{\mu(2Q)}{\mu(\rho Q)} \leq 1.
		\end{align*}
		For $I_2$, apply the third estimate of \cite[Lemma 2.1]{MPr}, Theorem \ref{thm2.2} and split the domain of integration into annuli to obtain
		\begin{align*}
			I_2 &\lesssim \frac{1}{\mu(\rho Q)}\int_Q\bigg( \int_{\mathbb{R}^{n+1}\setminus{2Q}}\frac{|\overline{y}-\overline{x}_0|}{|\overline{z}-\overline{x}_0|^{n+1}} \text{d}\mu(\overline{z})\bigg) \text{d}\mu(\overline{y}) \leq  \ell(Q)\frac{\mu(Q)}{\mu(\rho Q)} \int_{\mathbb{R}^{n+1}\setminus{2Q}}  \frac{\text{d}\mu(\overline{z})}{|\overline{z}-\overline{x}_0|^{n+1}}\\
			&\leq  \ell(Q) \sum_{j=1}^{\infty}\int_{2^{j+1}Q\setminus{2^jQ}}\frac{\text{d}\mu(\overline{z})}{|\overline{z}-\overline{x}_0|^{n+1}}\lesssim \ell(Q)\sum_{j=1}^{\infty}\frac{(2^{j+1}\ell(Q))^{n}}{(2^j\ell(Q))^{n+1}}\lesssim\sum_{j=1}^\infty\frac{1}{2^j} = 1,
		\end{align*}
		and the desired result follows.
	\end{proof}
	The previous lemma allows us to prove the main result of this subsection:
	\begin{thm}
		\label{thm3.11}
		Let $E\subset \mathbb{R}^{n+1}$ be a compact subset. Then, if $\gamma_+:=\gamma_{\Theta^{1/2},+}$,
		\begin{equation*}
			\gamma_+(E)\approx \widetilde{\gamma}_+(E).
		\end{equation*}
	\end{thm}
	\begin{proof}
		It suffices to prove $\gamma_+(E)\lesssim \widetilde{\gamma}_+(E)$. To this end, consider $\mu$ an admissible measure for $\gamma_+(E)$, i.e. $\mu$ is a positive Borel regular measure supported on $E$ such that $\|P\ast \mu\|_{\infty}\leq 1$. We know by Theorem \ref{thm2.2} that $\mu$ has $n$-growth with an absolute dimensional constant $C>0$. Hence, up to such constant, $\mu \in \Sigma(E)$. Moreover, by relation \eqref{eq3.2} and Lemma \ref{lem3.10} we have, for any $\rho \geq 2$,
		\begin{enumerate}[nolistsep]
			\item[\textit{1}.] $\pazocal{P}_\mu 1\in L^\infty(\mu)$ and thus $\pazocal{P}_\mu 1 \in \text{BMO}_{\rho}(\mu)$,
			\item[\textit{2}.] $\pazocal{P}^{\ast}_\mu 1 \in \text{BMO}_{\rho}(\mu)$,
			\item[\textit{3}.] $\pazocal{P}_\mu 1$ is $\mu$-weakly bounded.
		\end{enumerate}
		Applying a suitable $T1$-theorem, namely \cite[Theorem 1.3]{To2}, we deduce $\|\pazocal{P}_\mu\|_{L^2(\mu)\to L^2(\mu)}\lesssim 1$. This implies (again, up to a dimensional constant) that $\mu$ is admissible for $\gamma_{2,+}(E)$; and by the arbitrariness of $\mu$ we have $\gamma_{+}(E)\lesssim \gamma_{2,+}(E)$. So we conclude, by Theorem \ref{thm3.4},
		\begin{equation*}
			\gamma_{+}(E)\lesssim \gamma_{2,+}(E)\approx \widetilde{\gamma}_+(E).
		\end{equation*}
	\end{proof}
	A particular consequence of the last result is
	\begin{equation*}
		\gamma_{\Theta^{1/2},+}(E)\approx \gamma_{\text{\normalfont{sy}},+}(E) = \sup_{\mu}\Big\{ \mu(E)\;:\; \mu\in \Sigma(E),\; \|P_{\text{\normalfont{sy}}}\ast \mu\|_{\infty}\leq 1  \Big\}.
	\end{equation*}
	Since the same proof of Proposition \ref{prop3.2} yields that $\gamma_{\text{\normalfont{sy}},+}$ is invariant under temporal reflections, we also obtain
	\begin{equation*}
		\gamma_{\Theta^{1/2},+}(E) \approx \gamma_{\text{\normalfont{sy}},+}(E)=\gamma_{\text{\normalfont{sy}},+}(\pazocal{R}_t(E))\approx  \gamma_{\Theta^{1/2},+}(\pazocal{R}_t(E)) =  \gamma_{\overline{\Theta}^{1/2},+}(E).
	\end{equation*}
	
	\begin{cor}
		\label{cor3.12}
		The capacities $ \gamma_{\Theta^{1/2},+}, \gamma_{\overline{\Theta}^{1/2},+}$ and $\gamma_{\text{\normalfont{sy}},+}$ are comparable.
	\end{cor}
	
	\bigskip
	\section{The \mathinhead{\gamma_{\Theta^{1/2},+}}{} capacity of Cantor sets}
	\label{sec4}
	\bigskip
	
	Let us move on to finally estimate the ($1/2,+$)-caloric capacity of a certain family of Cantor sets, which generalize the particular example given in \cite[\textsection 5]{MPr}. In such example we are presented with a Cantor set $E$ inspired by the one constructed in \cite[Chapter IV, \textsection 2]{Ga} with positive $\pazocal{H}^n$-measure and which is removable for the $\Theta^{1/2}$-equation, meaning that $\gamma_{\Theta^{1/2}}(E)=0$. Our goal will be to generalize the latter and study its $\gamma_{\Theta^{1/2},+}$ capacity.\medskip\\
	Let $\lambda=(\lambda_j)_j$ be a sequence of real numbers satisfying $0< \lambda_j < 1/2$. We shall define its associated Cantor set $E\subset \mathbb{R}^{n+1}$ by the following algorithm. Set $Q^0:=[0,1]^{n+1}$ the unit cube of $\mathbb{R}^{n+1}$ and consider $2^{n+1}$ disjoint cubes inside $Q^0$ of side length $\ell_1:=\lambda_1$, with sides parallel to the coordinate axes and such that each cube contains a vertex of $Q^0$. Continue this same process now for each of the $2^{n+1}$ cubes from the previous step, but now using a contraction factor $\lambda_2$. That is, we end up with $2^{2(n+1)}$ cubes with side length $\ell_2:=\lambda_1\lambda_2$.\medskip\\
	It is clear that proceeding inductively we have that at the $k$-th step of the iteration we encounter $2^{k(n+1)}$ cubes, that we denote $Q_j^k$ for $1\leq j \leq 2^{k(n+1)}$, with side length $\ell_k:=\prod_{j=1}^k \lambda_j$. We will refer to them as cubes of the $k$-\textit{th generation}. We define
	\begin{equation*}
		E_k=E(\lambda_1,\ldots,\lambda_k):=\bigcup_{j=1}^{2^{k(n+1)}}Q_j^k,
	\end{equation*}
	and from the latter we obtain the Cantor set associated with $\lambda$,
	\begin{equation}
    \label{eq4.1}
		E=E(\lambda):=\bigcap_{k=1}^\infty E_k.
	\end{equation}
	If we chose $\lambda_j=2^{-(n+1)/n}$ for every $j$ we would recover the particular Cantor set presented in \cite[\textsection 5]{MPr}. We are first concerned with studying the Hausdorff dimension of $E$ in terms of $\lambda$, which we want it to be $n$, the critical dimension for $\gamma_{\Theta^{1/2}}$. An example of condition one imposes in such sequence to ensure the latter, as it is done in \cite{MTo} or \cite{To3}, is the following
	\begin{equation*}
		\lim_{k\to\infty} \ell_k\,2^{k(n+1)/n}=1.
	\end{equation*}
	
	Observe that a particular consequence of the previous equality is that there exists a constant $C>0$ depending on $\lambda$, so that
	\begin{equation*}
		\ell_k2^{k(n+1)/n}\geq C, \hspace{0.5cm} \forall k\geq 1.
	\end{equation*}
	Using such property one deduces $\pazocal{H}^n(E)>0$. Indeed, consider the probability measure $\mu$ defined on $E$ such that, for each generation $k$, $\mu(Q_j^k):=2^{-k(n+1)},\, 1 \leq j \leq 2^{k(n+1)}$. Let $Q$ be any cube, that we may assume to be contained in $Q^0$, and pick $k$ with the property $\ell_{k+1}\leq \ell(Q)\leq \ell_k$, so that $Q$ can meet, at most, $2^{n+1}$ cubes $Q_j^k$. Thus $\mu(Q)\leq 2^{-(k-1)(n+1)}$ and we deduce
	\begin{equation}
		\label{eq4.2}
		\mu(Q)\leq \frac{2^{2(n+1)}}{2^{(k+1)(n+1)}}=2^{2(n+1)}\Big( \ell_{k+1}2^{(k+1)(n+1)/n} \Big)^{-n}\ell_{k+1}^n\leq \frac{2^{2(n+1)}}{C^n}\ell(Q)^n\simeq \ell(Q)^n,
	\end{equation}
	meaning that $\mu$ presents $n$-growth. Therefore, by \cite[Chapter IV, Lemma 2.1]{Ga}, which follows from Frostman's lemma, we get $\pazocal{H}^n(E)\geq \pazocal{H}^n_\infty(E)>0$.\medskip\\
	Moreover, observe that for a fixed $0<\delta\ll 1$, there is $k$ large enough so that $\text{diam}(Q_j^k)\leq \delta$ and $ \ell_k\,2^{k(n+1)/n}\leq 2$. Thus, as $E_k$ defines a covering of $E$ admissible for $\pazocal{H}^n_\delta$, we get
	\begin{align*}
		\pazocal{H}^n_\delta(E)\leq \sum_{j=1}^{2^{k(n+1)}}\text{diam}(Q^k_j)^n\simeq \ell_k^n\,2^{k(n+1)}=\big( \ell_k\,2^{k(n+1)/n} \big)^n\leq 2^n.
	\end{align*}
	Since this procedure can be done for any $\delta$, we also have $\pazocal{H}^n(E)<\infty$ and thus $\text{dim}_\pazocal{H}(E)=n$ as wished.\medskip\\
	In order to state in a more compact way the results we are interested in, we introduce the following density for each $k\geq 1$,
	\begin{equation*}
		\theta_k:=\frac{1}{\ell_k^n\, 2^{k(n+1)}}= \frac{\mu(Q^k)}{\ell_k^n},
	\end{equation*}
	where $\mu$ is the probability measure defined above and $Q^k$ is any cube of the $k$-th generation. We also set $\theta_0:=1$.
	
	\begin{thm}
		\label{thm3.14}
		Let $(\lambda_j)_j$ be a sequence of real numbers satisfying $0<\lambda_j\leq \tau_0<1/2$ for every $j$, and $E$ its associated Cantor set as in \eqref{eq4.1}. Then, for each generation $k$,
		\begin{equation*}
			\gamma_{\Theta^{1/2},+}(E_k)\lesssim  \Bigg( \sum_{j=0}^{k}\theta_j  \Bigg)^{-1},
		\end{equation*}
		where the implicit constant in the relation $\lesssim$ only depends on $n$ and $\tau_0$.
	\end{thm}
	\begin{proof}
		Fix a positive integer $k$ and $\mu$ the probability measure supported on $E_k$ defined as
		\begin{equation*}
			\mu_k:=\frac{1}{|E_{k}|}\pazocal{L}^{n+1}|_{E_k}:=\frac{1}{\pazocal{L}^{n+1}(E_k)}\pazocal{L}^{n+1}|_{E_k}.
		\end{equation*}
		Observe that for every cube $Q^j_i$ of the $j$-th generation, with $0\leq j \leq k$ and $1\leq i \leq 2^{j(n+1)}$,
		\begin{equation*}
			\mu_k(Q_i^j)= \frac{1}{2^{k(n+1)}\ell_k^{n+1}}2^{(k-j)(n+1)}\ell_k^{n+1}=2^{-j(n+1)}.
		\end{equation*} 
		Fix any $\overline{x}=(x,t)\in E_k$ and consider the corresponding chain of cubes associated with $\overline{x}$,
		\begin{equation*}
			\overline{x}\in \Delta_{k}\subset\Delta_{k-1}\subset \cdots \subset \Delta_{1}\subset \Delta_{0}=:Q^0,
		\end{equation*}
		where $\Delta_j$ is the unique cube of the family $E_j$ that contains $\overline{x}$. Observe that
		\begin{align*}
			\int_{E_k}P_{\text{\normalfont{sy}}}(\overline{y}-\overline{x})&\text{d}\mu_k(\overline{y})=\int_{\Delta_{0}}P_{\text{\normalfont{sy}}}(\overline{y}-\overline{x})\text{d}\mu_k(\overline{y})\\
			&= \sum_{j=0}^{k-1}\int_{\Delta_{j}\setminus{\Delta_{j+1}}}P_{\text{\normalfont{sy}}}(\overline{y}-\overline{x})\text{d}\mu_k(\overline{y})+\int_{\Delta_k}P_{\text{\normalfont{sy}}}(\overline{y}-\overline{x})\text{d}\mu_k(\overline{y}) =:\sum_{j=1}^{k-1}I_{j} + I_k.
		\end{align*}
		If for each $0\leq j \leq k-1$ we write $\widetilde{\Delta}_{j+1}$ the cube of $E_{j+1}$ contained in $\Delta_{j}$ diagonally opposite to $\Delta_{j+1}$,
		\begin{align*}
			I_{j}&:=\int_{\Delta_{j}\setminus{\Delta_{j+1}}}\frac{|s-t|}{2|\overline{y}-\overline{x}|^{n+1}}d\mu_k(\overline{y})\geq \int_{ \widetilde{\Delta}_{j+1}}\frac{|s-t|}{2|\overline{y}-\overline{x}|^{n+1}}\text{d}\mu_k(\overline{y})\\
			&\gtrsim \frac{\ell_{j}-2\ell_{j+1}}{\ell_{j}^{n+1}}\mu_k(\widetilde{\Delta}_{j+1})= \frac{1}{\ell_{j}^n}\big(1-2\lambda_{j+1}\big)2^{-(j+1)(n+1)}\gtrsim \theta_{j}(1-2\tau_0)\simeq \theta_{j}.
		\end{align*}
		Regarding $I_k$, consider the cube $Q_{\overline{x}}:=Q(\overline{x},2\,\text{diam}(\Delta_k))$, that clearly contains $\Delta_k$, and for each positive integer $j$ write
		\begin{equation*}
			F_j:=Q_{\overline{x}} \cap \big\{ (y,s)\,:\; |t-s|>2^{-j}\text{diam}(\Delta_k) \big\},
		\end{equation*}
		as well as $F_0:=\varnothing$. Set $\widehat{F}_j:=F_{j+1}\setminus{F_j}$, and notice that $\{\widehat{F}_{j}\}_{j\geq 0}$ is a disjoint open covering of $Q_{\overline{x}}$. Therefore, since $\{\widehat{F}_{j}\cap \Delta_k\}_{j\geq 0}$ is also a disjoint covering $\Delta_k$,
		\begin{align*}
			I_k &= \frac{1}{|E_k|} \int_{\Delta_k}\frac{|t-s|}{2|\overline{x}-\overline{y}|^{n+1}}\text{d}\pazocal{L}^{n+1}(\overline{y}) \gtrsim \frac{1}{|E_k|\ell_k^{n+1}}\sum_{j=0}^{\infty} \int_{\widehat{F}_{j}\cap \Delta_k}|t-s|\text{d}\pazocal{L}^{n+1}(\overline{y})\\
			&\gtrsim \frac{1}{|E_k|\ell_k^{n}}\sum_{j=0}^{\infty} \frac{|\widehat{F}_{j}\cap \Delta_k|}{2^{j+1}}\simeq \frac{\ell_k}{|E_k|}\sum_{j=0}^{\infty} \frac{1}{2^{2j+1}} \simeq \frac{\ell_k}{|E_k|}=\frac{1}{2^{k(n+1)}\ell_k^n}=\theta_k.
		\end{align*}
		
		Thus, there exists some positive constant $C=C(n,\tau_0)$ such that
		\begin{equation*}
			\int_{E_k}P_{\text{\normalfont{sy}}}(\overline{y}-\overline{x})\text{d}\mu_k(\overline{y})\geq C^{-1} \sum_{j=0}^{k}\theta_{j}, \hspace{0.5cm} \forall \overline{x}\in E_{k}.
		\end{equation*}
		The last inequality can be rewritten as
		\begin{equation}
			\label{eq4.3}
			1\leq C\Bigg( \sum_{j=0}^{k}\theta_{j} \Bigg)^{-1}\int_{E_k}P_{\text{\normalfont{sy}}}(\overline{y}-\overline{x})\text{d}\mu_k(\overline{y}), \hspace{0.5cm} \forall \overline{x}\in E_{k}.
		\end{equation}
		Take $\nu$ any admissible measure for $\gamma_{\text{\normalfont{sy}},+}(E_k)$. Integrating on both sides of equation \eqref{eq4.3} and applying Tonelli's theorem together with $\|P_{\text{\normalfont{sy}}}\ast \nu\|_{\infty}\leq 1$,
		\begin{align*}
			\nu(E_{k})\leq C \Bigg(\sum_{j=0}^{k}\theta_j\Bigg)^{-1}\int_{E_k}P_{\text{\normalfont{sy}}}\ast\nu(\overline{y})\text{d}\mu_k(\overline{y})\leq C \Bigg(\sum_{j=0}^{k}\theta_j\Bigg)^{-1}.
		\end{align*}
		Hence, since $\nu$ was arbitrary, by Corollary \ref{cor3.12}, we finally conclude
		\begin{equation*}
			\gamma_{\Theta^{1/2},+}(E_k)\lesssim \gamma_{\text{\normalfont{sy}},+}(E_k) \leq  C  \Bigg( \sum_{j=0}^{k}\theta_j  \Bigg)^{-1}.
		\end{equation*}
	\end{proof}
	
	\begin{thm}
		\label{thm4.2}
		Let $(\lambda_j)_j$ be a sequence of real numbers satisfying $0<\lambda_j\leq \tau_0<1/2$, for every $j$. Then, for any fixed generation $k$,
		\begin{equation*}
			\gamma_{\Theta^{1/2},+}(E_k)\gtrsim  \Bigg( \sum_{j=0}^{k}\theta_j \Bigg)^{-1},
		\end{equation*}
		where the implicit constant of $\lesssim$ only depends on $n$ and $\tau_0$.
	\end{thm}
	\begin{proof}
		Fix a generation $k$ as well as the measure introduced in the proof of Theorem \ref{thm3.14},
		\begin{equation*}
			\mu_k:=\frac{1}{|E_k|}\pazocal{L}^{n+1}|_{E_k}.
		\end{equation*}
		Recall that $\mu_k(Q^j_i)=2^{-j(n+1)}$ for any cube of the $j$-th generation, with $0\leq j \leq k$ and $1\leq i \leq 2^{j(n+1)}$. Let us fix any $\overline{x}\in \mathbb{R}^{n+1}$ and proceed in an inductive way as follows: 
		\begin{itemize}
			\item[\textit{1}.] If $d(\overline{x}, Q^0)\geq 1$, it is clear that $P\ast \mu_k(\overline{x})\leq 1=\theta_0$. If $d(\overline{x}, Q^0)< 1$, denote by $\Delta_1$ one of the cubes of the first generation $E_1$ that is closest to $\overline{x}$. Observe that $d(\overline{x},E_1\setminus{\Delta_1})\gtrsim 1-2\lambda_1\geq 1-2\tau_0$. Then,
			\begin{align*}
				P\ast \mu_k(\overline{x})&= \int_{E_k}P(\overline{x}-\overline{y})\text{d}\mu_k(\overline{y}) \leq \int_{E_1\setminus{\Delta_1}}\frac{\text{d}\mu_k(\overline{y})}{|\overline{x}-\overline{y}|^n}+\int_{\Delta_1}\frac{\text{d}\mu_k(\overline{y})}{|\overline{x}-\overline{y}|^n}\\
				&\lesssim \frac{\theta_0}{(1-2\tau_0)^n} + \int_{\Delta_1}\frac{\text{d}\mu_k(\overline{y})}{|\overline{x}-\overline{y}|^n}.
			\end{align*}
			\item[\textit{2}.] If $d(\overline{x},\Delta_1)\geq \ell_1$, it is clear that the above remaining integral satisfies
			\begin{equation*}
				\int_{\Delta_1}\frac{\text{d}\mu_k(\overline{y})}{|\overline{x}-\overline{y}|^n} \leq \frac{\mu_k(\Delta_1)}{\ell_1^n}=\theta_1,
			\end{equation*}
			and therefore $P\ast \mu_k(\overline{x})\lesssim \theta_0+\theta_1$. On the other hand, if $d(\overline{x},\Delta_1)< \ell_1$, we repeat the process of step \textit{1} and pick $\Delta_2$ one of the cubes of $E_2$ that is closest to $\overline{x}$. In the current setting notice that $d(\overline{x},E_2\setminus{\Delta_2})\gtrsim (1-2\tau_0)\ell_1$, which implies
			\begin{align*}
				P\ast \mu_k(\overline{x}) &\lesssim \frac{\theta_0}{(1-2\tau_0)^n}+\int_{E_2\setminus{\Delta_2}}\frac{\text{d}\mu_k(\overline{y})}{|\overline{x}-\overline{y}|^n}+\int_{\Delta_2}\frac{\text{d}\mu_k(\overline{y})}{|\overline{x}-\overline{y}|^n}\\
				&\lesssim \frac{1}{(1-2\tau_0)^n}(\theta_0+\theta_1)+\int_{\Delta_2}\frac{\text{d}\mu_k(\overline{y})}{|\overline{x}-\overline{y}|^n}.
			\end{align*}
		\end{itemize}
		In general, for $1\leq m <k-1$, the $(m+1)$-th step of the above process we would begin by dealing with an estimate of the form
		\begin{equation*}
			P\ast \mu_k(\overline{x}) \lesssim \frac{1}{(1-2\tau_0)^n}(\theta_0+\theta_1+\cdots+\theta_{m-1})+\int_{\Delta_m}\frac{\text{d}\mu_k(\overline{y})}{|\overline{x}-\overline{y}|^n},
		\end{equation*}
		and we would distinguish whether if $d(\overline{x},\Delta_j)\geq \ell_j$ or not. If the inequality is satisfied, it follows that $P\ast \mu_k(\overline{x})\lesssim \sum_{j=0}^{m}\theta_j$. If on the other hand $d(\overline{x},\Delta_j)< \ell_j$, write $\Delta_{m+1}$ one of the cubes of $E_{m+1}$ that is closest to $\overline{x}$ and notice that $d(\overline{x},E_{m+1}\setminus{\Delta_{m+1}})\gtrsim (1-2\tau_0)\ell_m$, and move on to step $m+2$.
		The previous process can carry on a maximum of $k$ steps (this is the case, for example, if $\overline{x}\in E_k$), and in this situation $\overline{x}$ satisfies $d(\overline{x},\Delta_k)<\ell_k$ as well as the estimate 
		\begin{equation*}
			P\ast \mu_k(\overline{x})\lesssim \frac{1}{(1-2\tau_0)^n}\sum_{j=0}^{k-1}\theta_j +\int_{\Delta_k}\frac{\text{d}\mu_k(\overline{y})}{|\overline{x}-\overline{y}|^n},
		\end{equation*}
		that cannot	be dealt with the same iterative method. We name the remaining integral $I_k$ and write $Q_k:=Q\big(\overline{x},5\ell_k\big)$ so that $\Delta_k \subset Q_k$. We split the previous cube into the annuli $A_j:=Q\big( \overline{x},5\ell_k 2^{-j}\big)\setminus{Q\big(\overline{x},5\ell_k 2^{-j-1}\big)}$, for $j \geq 0$ integer. Therefore,
		\begin{align*}
			I_k&\leq \frac{1}{|E_k|}\sum_{j=0}^\infty\int_{A_j} \frac{\text{d}\pazocal{L}^{n+1}(\overline{y})}{|\overline{x}-\overline{y}|^n}
			\leq \frac{1}{|E_k|}\sum_{j=0}^\infty \frac{(5\,\ell_k2^{-j})^{n+1}}{(5\ell_k2^{-j-1})^{n}} \simeq \frac{\ell_k}{|E_k|}\sum_{j=0}^{\infty}\frac{1}{2^j} \simeq \theta_k.
		\end{align*}
		With this we conclude that, in general, there exists a constant $C=C(n,\tau_0)>0$ so that
		\begin{equation*}
			P\ast \mu_k(\overline{x})\Bigg(C\sum_{j=0}^{k}\theta_j\Bigg)^{-1}\leq 1, \hspace{0.5cm} \forall \overline{x}\in \mathbb{R}^{n+1}.
		\end{equation*}
		Hence, the measure $\big(C \sum_{j=0}^k \theta_k \big)^{-1}\mu_k$ is admissible for $\gamma_{\Theta^{1/2},+}(E_k)$, and the result follows.
	\end{proof}
	
	Combining both Theorems \ref{thm3.14} and \ref{thm4.2}, we obtain the following result:
	
	\begin{cor}
		\label{cor4.3}
		Let $(\lambda_j)_j$ be a sequence of real numbers satisfying $0<\lambda_j\leq \tau_0<1/2$, for every $j$. Then, if $E$ denotes the associated Cantor set as in \eqref{eq4.1},
        \begin{equation*}
			\gamma_{\Theta^{1/2},+}(E) \approx \Bigg( \sum_{j=0}^{\infty}\theta_j  \Bigg)^{-1}.
		\end{equation*}
		where the implicit constants only depend on $n$ and $\tau_0$.
	\end{cor}
	\begin{proof}
		The estimate follows from Theorem \ref{thm3.14} and the monotonicity of $\gamma_{\Theta^{1/2},+}$; and from Theorem \ref{thm4.2} combined with the third point of Proposition \ref{prop3.1} (outer regularity).
	\end{proof}
	
	\subsection{An additional estimate for \mathinhead{\gamma_{\Theta^{1/2}}}{}}
	\label{subsec4.1}
	\bigskip
	
	Let us present one last result of similar nature to \cite[Theorem 5.3]{MPr} that admits an analogous proof and concerns not only $\gamma_{\Theta^{1/2},+}$, but also $\gamma_{\Theta^{1/2}}$. To obtain it, however, we need to assume that the sequence $(\theta_k)_k$ is decreasing (which is equivalent to assuming that $2^{-\frac{n+1}{n}} \leq \lambda_k$ for every $k$) and $\theta_k\geq \kappa$, for some absolute constant $\kappa>0$ and every $k$. Notice that this last condition implies $\pazocal{H}^n(E)<\infty$. Indeed, for any fixed $0<\delta\ll 1$, we may pick a generation $k$ large enough so that
	\begin{equation*}
		\pazocal{H}_{\delta}^n(E)\leq \pazocal{H}_{\delta}^n(E_k) \lesssim 2^{k(n+1)}\ell_k^n = \theta_k^{-1} \leq \kappa^{-1},
	\end{equation*}
	meaning that $\pazocal{H}^n(E)<\infty$.
	\begin{thm}
		\label{thm4.4}
		Let $(\lambda_j)_j$ satisfy $0< 2^{-\frac{n+1}{n}} \leq \lambda_j <1/2$, for every $j$, and let $E$ be its associated Cantor set as in \eqref{eq4.1}. If $\theta_k\geq \kappa$, for some $\kappa>0$ and every $k$,
		\begin{equation*}
			\gamma_{\Theta^{1/2}}(E)=0.
		\end{equation*}
	\end{thm}
	\begin{proof}
		Since the sequence $(\theta_k)_k$ is non-increasing and $\theta_0=1$, we have $\kappa \leq \theta_k \leq 1$, for every $k$. Also notice that given $\mu$ the probability measure on $E$ with $\mu(Q^k_j)=2^{-k(n+1)}$ for every generation $k$ and $1\leq j \leq 2^{k(n+1)}$; by an analogous argument to that of \eqref{eq4.2}, we get that $\mu$ presents $n$-growth.
		Moreover, since we can assume without loss of generality that $0<\pazocal{H}^n(E)<\infty$, we deduce that $\mu$ coincides with $\pazocal{H}^n|_E$ modulo a constant factor.\medskip\\
		Let us proceed with the proof, which will be analogous to that of \cite[Theorem 5.3]{MPr}. Indeed, if we assume that $\gamma_{\Theta^{1/2}}(E)>0$, we are able to reach a contradiction as follows: by \cite[Theorem 5.1]{MPr} we may pick $\nu$ a signed measure supported on $E$ with $|\langle \nu , 1 \rangle |>1$, $\|P\ast \nu\|_{\infty}\leq 1$ and satisfying that there exists a Borel function $f\in L^\infty(\mu)$ such that
		\begin{equation*}
			\nu = f\mu.
		\end{equation*}
		The contradiction we are looking for arises from the bound \cite[Equation (5.7)]{MPr}, i.e.
		\begin{equation*}
			\widetilde{\pazocal{P}}_{\nu,\ast} (\overline{x}):=\sup_{k\geq 0} \big\rvert \pazocal{P}_\nu \chi_{\mathbb{R}^{n+1}\setminus{\Delta_k}}(\overline{x}) \big\rvert \leq K', \hspace{0.5cm} \forall \overline{x}\in E,
		\end{equation*}
		where recall that for any $\overline{x}\in E$, $\Delta_k$ is the unique cube of the family $E_k$ that contains $\overline{x}$. Proceeding exactly as in the proof of \cite[Theorem 5.3]{MPr}, we are able to reach the estimate
		\begin{equation*}
			\widetilde{\pazocal{P}}_{\nu,\ast}(\overline{z}) \gtrsim \Bigg( \sum_{h=k}^{k+m-1} \theta_h \Bigg) f(\overline{x}_0) - \varepsilon\,\Bigg(\prod_{j=k+1}^{k+m}\frac{1}{\lambda_j}\Bigg)^n\theta_k \geq \kappa (m-1) f(\overline{x}_0) - 2^{m(n+1)} \varepsilon,
		\end{equation*}
		where $\overline{z}=(z_1,\ldots,z_n,t)$ is one of the upper leftmost corners of $\Delta_k$ (with $z_1$ minimal and $t$ maximal in $\Delta_k$), $\overline{x}_0\in E$ is a Lebesgue point (with respect to $\mu$) for $f=\text{d}\nu/\text{d}\mu$ satisfying $f(\overline{x}_0)>0\,$; $m\gg 1$ is an integer parameter independent of $k$, and $0< \varepsilon\ll 1$ is a parameter that once fixed, also fixes the value of $k$. Therefore, by choosing $m$ large enough and then $\varepsilon$ small enough, we reach the desired contradiction.
	\end{proof} 
	
	\bigskip
	\section{The BMO variant of \mathinhead{\gamma_{\Theta^{1/2}}}{}}
	\label{sec5}
	\bigskip
	
	We devote the final sections of our paper to characterize the BMO and $\text{Lip}_\alpha$ variants of the $1/2$-caloric capacity. In the present section we study the former case, similarly as it is done in \cite[\textsection 13.5.1]{AIMar} for analytic capacity. That is, the main goal of this section is to give an analogous description of such object in terms of a particular Hausdorff content. To introduce it, let us recall the definition of the usual BMO space of $\mathbb{R}^{n+1}$ (rather than its generalization of Definition \ref{def3.2}).
	\begin{defn}[BMO]
		\label{def5.1}
		A function $f\in L^1_\text{loc}(\mathbb{R}^{n+1})$ belongs to the \textit{BMO} space if its BMO norm is finite, that is
		\begin{equation*}
			\|f\|_{\ast}:=\sup_{Q}\frac{1}{|Q|}\int_Q\big\rvert f(\overline{x}) -f_Q\big\rvert \text{d}\pazocal{L}^{n+1}(\overline{x}),
		\end{equation*}
		where $|Q|:=\pazocal{L}^{n+1}(Q)$ and $f_Q:=\frac{1}{|Q|}\int_Q f\, \text{d}\pazocal{L}^{n+1}$.
	\end{defn}
	
	We now introduce the BMO variant of the $1/2$-caloric capacity:
	
	\begin{defn}[BMO $1/2$-caloric capacity]
		\label{def5.2}
		Given a compact subset $E\subset \mathbb{R}^{n+1}$, define its \textit{BMO $1/2$-caloric capacity} as
		\begin{equation*}
			\gamma_{\Theta^{1/2},\ast}(E)=\sup |\langle T, 1 \rangle| ,
		\end{equation*}
		where the supremum is taken among all distributions $T$ with $\text{supp}(T)\subseteq E$ satisfying
		\begin{equation*}
			\|P\ast T\|_{\ast} \leq 1
		\end{equation*}
		Such distributions will be called \textit{admissible for $\gamma_{\Theta^{1/2},\ast}(E)$}.
	\end{defn}
	
	Since $\gamma_{\Theta^{1/2}}\leq \gamma_{\Theta^{1/2},\ast}$, Theorem \ref{thm2.1} implies that if $\gamma_{\Theta^{1/2}, \ast}(E)=0$, then $E$ is $1/2$-caloric removable.
	
	\bigskip
	\subsection{Comparability of \mathinhead{\gamma_{\Theta^{1/2},\ast}}{} to the Hausdorff measure}
	\label{subsec5.1}
	\bigskip
	
	We begin by noticing that distributions admissible for the BMO $1/2$-caloric capacity exhibit the same growth condition to that described in Theorem \ref{thm2.2}.
	
	\begin{thm}
		\label{thm5.1}
		Let $T$ be a distribution in $\mathbb{R}^{n+1}$ with $\|P\ast T\|_\ast\leq 1$. If $\varphi$ is a $\pazocal{C}^1$ function supported on $Q\subset \mathbb{R}^{n+1}$ with $\|\nabla \varphi\|_\infty \leq \ell(Q)^{-1}$, then 
		\begin{equation*}
			|\langle T, \varphi \rangle |\lesssim \ell(Q)^n.
		\end{equation*}
	\end{thm}
	
	\begin{proof}
		Let $T$ and $\varphi$ satisfy the conditions of the statement. Since $P$ is the fundamental solution of $\Theta^{1/2}$,
		\begin{align*}
			|\langle T, \varphi\rangle|&=|\langle \Theta^{1/2}(P\ast T), \varphi\rangle|\\
			&\leq|\langle P\ast T-(P\ast T)_{2Q}, (-\Delta)^{1/2}\varphi\rangle|+|\langle P\ast T-(P\ast T)_{Q}, \partial_t\varphi \rangle|=:I_1+I_2.
		\end{align*}
		Regarding $I_2$, use the normalization condition $\|\partial_t\varphi\|_\infty\leq \|\nabla\varphi\|_\infty\leq \ell(Q)^{-1}$ to deduce
		\begin{align*}
			I_2&\leq \int_{Q}\big\rvert P\ast T(\overline{x})-(P\ast T)_{Q}\big\rvert \big\rvert\partial_t\varphi(\overline{x})\big\rvert \text{d}\pazocal{L}^{n+1}(\overline{x})\\
			&\leq \ell(Q)^{n+1}\bigg(\frac{1}{|Q|}\int_{Q}\big\rvert P\ast T(\overline{x})-(P\ast T)_{Q}\big\rvert \big\rvert \partial_t\varphi(\overline{x})\big\rvert \text{d}\pazocal{L}^{n+1}(\overline{x})\bigg) \\
			&\leq \|P\ast T\|_{\ast}\,\ell(Q)^{n}\leq \ell(Q)^{n}.
		\end{align*}
		For $I_1$, write $Q=Q_1\times I_Q\subset \mathbb{R}^{n}\times \mathbb{R}$ and observe that
		\begin{align*}
			I_1\leq \int_{2Q_1\times I_Q}&\big\rvert P\ast T(\overline{x})-(P\ast T)_{2Q} \big\rvert  \big\rvert (-\Delta)^{1/2}\varphi(\overline{x}) \big\rvert \text{d}\pazocal{L}^{n+1}(\overline{x})\\
			&+\int_{(\mathbb{R}^n\setminus{2Q_1})\times I_Q}\big\rvert P\ast T(\overline{x})-(P\ast T)_{2Q} \big\rvert \big\rvert (-\Delta)^{1/2}\varphi(\overline{x}) \big\rvert \text{d}\pazocal{L}^{n+1}(\overline{x})\\
			&=:I_{11}+I_{12}.
		\end{align*}
		
		To deal with $I_{11}$ we apply the Cauchy-Schwarz inequality,
		\begin{align*}
			I_{11}&\leq \int_{2Q}\big\rvert P\ast T(\overline{x})-(P\ast T)_{2Q}\big\rvert \big\rvert (-\Delta_x)^{1/2}\varphi(\overline{x})\big\rvert \text{d}\pazocal{L}^{n+1}(\overline{x})\\
			&\leq \bigg( \int_{2Q} \big\rvert P\ast T(\overline{x})-(P\ast T)_{2Q}\big\rvert^2 \text{d}\pazocal{L}^{n+1}(\overline{x}) \bigg)^{1/2}\bigg( \int_{2Q} \big\rvert (-\Delta)^{1/2}\varphi(\overline{x})\big\rvert^2 \text{d}\pazocal{L}^{n+1}(\overline{x}) \bigg)^{1/2}.
		\end{align*}
		Observe that by John-Nirenberg's inequality \cite[Corollary 6.12]{Du}, the first factor satisfies
		\begin{equation*}
			\bigg( \int_{2Q} \big\rvert P\ast T(\overline{x})-(P\ast T)_{2Q}\big\rvert^2 \text{d}\pazocal{L}^{n+1}(\overline{x}) \bigg)^{1/2} \lesssim \ell(Q)^{(n+1)/2}\|P\ast T\|_\ast\leq \ell(Q)^{(n+1)/2}.
		\end{equation*}
		On the other hand, concerning the second one recall that
		\begin{equation*}
			(-\Delta)^{1/2}\varphi = (-\Delta_x)^{1/2}\varphi \approx \sum_{j=1}^N R_j\partial_j \varphi,
		\end{equation*}
		with $R_j$, $1\leq j \leq n$, being the Riesz transforms with Fourier multiplier $\xi_j/|\xi|$. Since these operators are bounded on $L^2$,
		\begin{align*}
			\bigg( \int_{2Q} \big\rvert (-\Delta)^{1/2}\varphi(\overline{x})\big\rvert^2 \text{d}\pazocal{L}^{n+1}(\overline{x}) \bigg)^{1/2}&\lesssim \sum_{j=1}^n\|R_j\partial_j \varphi\|_{L^2(2Q)}\\
			& \lesssim \sum_{j=1}^n\|\partial_j\varphi\|_{L^2(2Q)}\lesssim \frac{\ell(Q)^{(n+1)/2}}{\ell(Q)},
		\end{align*}
		where in the last step we have applied the normalization estimate $\|\partial_j \varphi\|_\infty \leq \|\nabla \varphi\|_\infty\leq \ell(Q)^{-1}$. Therefore, combining the bounds for both factors we finally get $I_{11}\lesssim \ell(Q)^n$.	Regarding $I_{12}$ let us name $f:=P\ast T$ so that
		\begin{align*}
			I_{12}&\leq \int_{(\mathbb{R}^n\setminus{2Q_1})\times 2I_Q}\big\rvert f(\overline{x})-f_{2Q} \big\rvert \big\rvert (-\Delta)^{1/2}\varphi(\overline{x}) \big\rvert \text{d}\pazocal{L}^{n+1}(\overline{x})\\
			&=\sum_{j=1}^\infty \int_{C_{j+1}\setminus{C_j}} \big\rvert f(\overline{x})-f_{2Q} \big\rvert \big\rvert (-\Delta)^{1/2}\varphi(\overline{x}) \big\rvert \text{d}\pazocal{L}^{n+1}(\overline{x}),
		\end{align*}
		where we have defined the cylinders $C_j:=2^jQ_1\times I_{2Q}$ for $j\geq 1$. Continue by observing that since $\varphi$ is supported on $Q$, by the divergence theorem (see \cite[A8.8]{Alt}, for example) it is clear that $\int_{Q_1}\partial_j\varphi (z,t)d\pazocal{L}^n(z)=0$, for each $t\in I_Q$. Therefore, for any $\overline{x}\notin 2Q_1\times I_{2Q}$, if $x_0\in \mathbb{R}^n$ denotes the center of $Q_1$,
		\begin{align*}
			\big\rvert (-\Delta)^{1/2}\varphi(\overline{x}) \big\rvert &\leq \sum_{j=1}^{n}\big\rvert R_j\partial_j \varphi (\overline{x}) \big\rvert = \sum_{j=1}^{n}\bigg\rvert \int_{Q_1} \partial_j\varphi(z,t)\frac{z_j-x_j}{|z-x|^{n+1}} \bigg\rvert\text{d}\pazocal{L}^n(z)\\
			&=\sum_{j=1}^{n}\bigg\rvert \int_{Q_1} \partial_j\varphi(z,t)\bigg(\frac{z_j-x_j}{|z-x|^{n+1}}-\frac{x_{0,j}-x_j}{|x_{0}-x|^{n+1}}\bigg) \bigg\rvert\text{d}\pazocal{L}^n(z)\\
			&\lesssim \sum_{j=1}^{n}\int_{Q_1}\big\rvert\partial_j\varphi(z,t)\big\rvert \frac{|z_j-x_{0,j}|}{|\widetilde{z}-x|^{n+1}}\text{d}\pazocal{L}^n(z)\lesssim \sum_{j=1}^{n} \frac{\ell(Q)}{|x_{0}-x|^{n+1}}\|\nabla_x\varphi\|_{\infty}\ell(Q)^n\\
			&\lesssim \frac{\ell(Q)^{n}}{|x_{0}-x|^{n+1}},
		\end{align*}
		where we have applied the mean value theorem so that $\widetilde{z}\in Q_1$ depends on $z$. Notice also that $|\widetilde{z}-x|\approx |x_{0}-x|$, since $x\notin 2Q_1$. This way we obtain
		\begin{align*}
			I_{12}&\lesssim \sum_{j=1}^\infty \frac{\ell(Q)^n}{(2^j\ell(Q))^{n+1}} \int_{C_{j+1}\setminus{C_j}} \big\rvert f(\overline{x})-f_{2Q} \big\rvert \text{d}\pazocal{L}^{n+1}(\overline{x})\\
			&\hspace{-0.35cm}\leq \frac{1}{\ell(Q)} \sum_{j=1}^\infty \frac{1}{2^{j(n+1)}}\bigg( \int_{C_{j+1}\setminus{C_j}} \big\rvert f(\overline{x})-f_{2^jQ} \big\rvert \text{d}\pazocal{L}^{n+1}(\overline{x})+\int_{C_{j+1}\setminus{C_j}} \big\rvert f_{2Q}-f_{2^jQ} \big\rvert \text{d}\pazocal{L}^{n+1}(\overline{x})\bigg).
		\end{align*}
		For the first integral we shall apply Hölder's inequality for some exponent $q$, that will be fixed later on; as well as John-Nirenberg's inequality,
		\begin{align*}
			\int_{C_{j+1}\setminus{C_j}}\big\rvert f(\overline{x})-f_{2^jQ}\big\rvert &\text{d}\pazocal{L}^{n+1}(\overline{x})\\
			&\leq \bigg( \int_{2^jQ} \big\rvert f(\overline{x})-f_{2^jQ}\big\rvert^q d\pazocal{L}^{n+1}(\overline{x}) \bigg)^{1/q}\pazocal{L}^{n+1}(C_{j+1}\setminus{C_j})^{1/q'}\\
			&\lesssim \big(2^j\ell(Q)\big)^{(n+1)/q}\|f\|_{\ast}\big( 2\ell(Q)\,2^{nj}\ell(Q)^n \big)^{1/q'}\\
			&\leq \ell(Q)^{n+1}\,2^{j(n+1/q)+1/q'}.
		\end{align*}
		For the second integral apply, for example, \cite[Proposition 6.5]{Du} to deduce that $f_{2^jQ}$ and $f_{Q}$ are majored by $\|f\|_{\ast}\leq 1$. Thus,
		\begin{equation*}
			\int_{C_{j+1}\setminus{C_j}}\big\rvert f_{2^jQ}-f_Q\big\rvert \text{d}\pazocal{L}^{n+1}(\overline{x})\lesssim \pazocal{L}^{n+1}(C_{j+1}\setminus{C_j})\lesssim \ell(Q)^{n+1}\,2^{jn+1}.
		\end{equation*}
		All in all we obtain
		\begin{equation*}
			I_{12}\lesssim \ell(Q)^n\sum_{j=1}^\infty \frac{1}{2^{j(n+1)}}\Big( 2^{j(n+1/q)+1/q'}+2^{jn+1} \Big) \lesssim \ell(Q)^n\bigg(1+\sum_{j=1}^\infty \frac{2^{1/q'}}{2^{j(1-1/q)}}\bigg).
		\end{equation*}
		This last sum is convergent if and only if $q>1$. So fixing an exponent satisfying this last condition we deduce the result.
	\end{proof}
	
	The previous growth result combined with \cite[Lemma 5.2]{MPr} yields the following:
	\begin{thm}
		\label{thm5.2}
		Let $E\subset \mathbb{R}^{n+1}$ be a compact subset with $\pazocal{H}^{n}(E)<\infty$ and $T$ an admissible distribution for $\gamma_{\Theta^{1/2},\ast}(E)$. Then $T$ is a signed measure which is absolutely continuous with respect to $\pazocal{H}^{n}|_E$ and there exists a Borel function $f:E\to \mathbb{R}$ such that $T=f\,\pazocal{H}^{n}|_E$ and that satisfies $\|f\|_{L^\infty(\pazocal{H}^{n}|_E)}\lesssim 1$.
	\end{thm}
	
	Let us turn to the result we are interested in. Its statement reads as follows:
	\begin{thm}
		\label{thm5.3}
		There are $($dimensional$\,)$ constants $C_1,C_2>0$ so that for any $E\subset \mathbb{R}^{n+1}$ compact,
		\begin{equation*}
			C_1 \pazocal{H}^{n}_{\infty}(E)\leq \gamma_{\Theta^{1/2},\ast}(E)\leq C_2\pazocal{H}_{\infty}^{n}(E).
		\end{equation*}
	\end{thm}
	
	\begin{proof}
		Let us focus first on the right-hand side inequality. Although it seems that it would follow directly from Theorem \ref{thm5.2}, we shall give an standard argument based only on Theorem \ref{thm5.1} to avoid the possible dependence of $f$ with respect to $T$. Let us proceed then by fixing $\varepsilon>0$ and $\{A_k\}_k$ a collection of sets in $\mathbb{R}^{n+1}$ that cover $E$ such that
		\begin{equation*}
			\sum_{k=1}^\infty \text{diam}(A_k)^{n}\leq \pazocal{H}^{n}_{\infty}(E)+\varepsilon.
		\end{equation*}
		For each $k$ let $Q_k$ an open cube centered at some point $a_k\in A_k$ with side length $\ell(Q_k)=\text{diam}(A_k)$, so that $E\subset \bigcup_kQ_k$. By compactness, we can assume this last open covering to be finite. We denote it $\{Q_1,\ldots, Q_N\}$. By the usual Harvey-Polking lemma \cite[Lemma 3.1]{HPo} there exists a collection of smooth functions $\{\varphi_k\}_{k=1}^N$ such that
        $\sum_{k=1}^N\varphi_k = 1$ in $\bigcup_{k=1}^N Q_k$ and $0\leq \varphi_k\leq 1$, $\text{supp}(\varphi_k)\subset 2Q_k$, $\|\nabla\varphi_k\|_\infty\leq \ell(Q_k)^{-1}$, for each $1\leq k \leq N$. By Theorem \ref{thm5.1}, if $T$ is any distribution admissible for $\gamma_{\Theta^{1/2},\ast}(E)$,\newpage
		\begin{align*}
			|\langle T, 1 \rangle |=\bigg\rvert \sum_{k=1}^N \langle T, \varphi_k \rangle \bigg\rvert  \lesssim \sum_{k=1}^N \ell(Q_k)^{n} = \sum_{k=1}^N\text{diam}(A_k)^{n}\leq \pazocal{H}^{n}_{\infty}(E)+\varepsilon.
		\end{align*}
		Since this holds for any $T$ and $\varepsilon>0$ can be arbitrarily small, the inequality follows.\medskip\\
		For the left-hand side inequality we will apply Frostman's lemma \cite[Theorem 8.8]{Ma}. Assume then $\pazocal{H}_{\infty}^{n}(E)>0$ and consider a non trivial positive Borel regular measure $\mu$ supported on $E$ with $\mu(E)\geq c\pazocal{H}_{\infty}^{n}(E)$ and $\mu(B(\overline{x},r))\leq r^{n}$ for all $\overline{x}\in\mathbb{R}^{n+1}$, $r>0$. If we prove that
		\begin{equation*}
			\|P\ast \mu\|_{\ast}\lesssim 1,
		\end{equation*}
		we will be done, since this would imply $\gamma_{\Theta^{1/2}, \ast}(E)\gtrsim \langle \mu, 1 \rangle = \mu(E) \gtrsim \pazocal{H}_{\infty}^{n}(E)$. To control the BMO norm of $P \ast \mu$ we proceed exactly as in Lemma \ref{lem3.10}: fix a cube $Q=Q(\overline{x}_0,\ell(Q))$ and consider $\chi_{2Q}$ together $\chi_{2Q^c}:=1-\chi_{2Q}$. Choose the constant
		\begin{equation*}
			c_Q:=P\ast (\chi_{2Q^c}\mu)(\overline{x}_0),  
		\end{equation*}
		and observe that
		\begin{align*}
			\frac{1}{|Q|}\int_Q &|P\ast \mu(\overline{y})-c_Q|\text{d}\pazocal{L}^{n+1}(\overline{y})\\
			&\leq \frac{1}{|Q|}\int_Q\bigg(\int_{2Q}P(\overline{y}-\overline{z})\text{d}\mu(\overline{z})\bigg)\text{d}\pazocal{L}^{n+1}(\overline{y})\\
			&\hspace{1.4cm}+\frac{1}{|Q|}\int_Q\bigg( \int_{\mathbb{R}^{n+1}\setminus{2Q}}|P(\overline{y}-\overline{z})-P(\overline{x}_0-\overline{z})| \text{d}\mu(\overline{z})\bigg) \text{d}\pazocal{L}^{n+1}(\overline{y}) =:I_1+I_2.
		\end{align*}
		Regarding $I_1$, after applying Tonelli's theorem we may directly integrate using polar coordinates to obtain
		\begin{align*}
			I_1\lesssim \frac{1}{|Q|}\int_{2Q}\bigg( \int_{Q}\frac{\text{d}\pazocal{L}^{n+1}(\overline{y})}{|\overline{y}-\overline{z}|^{n}} \bigg)\text{d}\mu(\overline{z}) \lesssim \frac{\ell(Q)\mu(2Q)}{|Q|}\lesssim 1.
		\end{align*}
		where the last step is due to the fact that $\mu$ has $n$-growth with constant $1$. Turning to $I_2$, the third estimate of \cite[Lemma 2.1]{MPr} yields
		\begin{align*}
			I_2 &\lesssim \frac{1}{|Q|}\int_Q\bigg( \int_{\mathbb{R}^{n+1}\setminus{2Q}}\frac{|\overline{y}-\overline{x}_0|}{|\overline{z}-\overline{x}_0|^{n+1}} \text{d}\mu(\overline{z})\bigg) \text{d}\pazocal{L}^{n+1}(\overline{y}) \\
			&\leq  \ell(Q) \int_{\mathbb{R}^{n+1}\setminus{2Q}}  \frac{\text{d}\mu(\overline{z})}{|\overline{z}-\overline{x}_0|^{n+1}} =\ell(Q)\sum_{j=1}^{\infty}\int_{2^{j+1}Q\setminus{2^jQ}}\frac{\text{d}\mu(\overline{z})}{|\overline{z}-\overline{x}_0|^{n+1}}\\
			&\hspace{5.05cm}\lesssim \ell(Q)\sum_{j=1}^{\infty}\frac{(2^{j+1}\ell(Q))^{n}}{(2^j\ell(Q))^{n+1}}\lesssim\sum_{j=1}^\infty\frac{1}{2^j}= 1,
		\end{align*}
		and so the desired result follows.
	\end{proof}
	
	\bigskip
	\section{The \mathinhead{\text{Lip}_\alpha}{} variant of \mathinhead{\gamma_{\Theta^{1/2}}}{}}
	\label{sec6}
	\bigskip
	
	In this last section we shall study a variant of the capacity $\gamma_{\Theta^{1/2}}$ defined through a normalization condition that involves a $\text{Lip}_\alpha$ seminorm. This study has been motivated by the one carried out in \cite{Me} for analytic capacity. We also clarify that, in this subsection, the symbols $\simeq$ and $\lesssim, \gtrsim$ will denote equalities and inequalities respectively, with implicit constants that may depend on the dimension $n$ and also the parameter $\alpha$. Let us begin by reminding a basic definition:
	\begin{defn}[$\text{Lip}_\alpha$ function]
		\label{def6.1}
		A function $f:\mathbb{R}^{n+1}\to\mathbb{R}$ is  $\textit{\text{Lip}}_\alpha$ for some $0<\alpha<1$ if
		\begin{equation*}
			\|f\|_{\text{Lip}_\alpha}:=\sup_{\overline{x},\overline{y}\in\mathbb{R}^{n+1}}\frac{|f(\overline{x})-f(\overline{y})|}{|\overline{x}-\overline{y}|^{\alpha}}\lesssim 1.
		\end{equation*}
	\end{defn}
	
	\begin{defn}[$\text{Lip}_\alpha$ $1/2$-caloric capacity]
		\label{def6.2}
		Given a compact subset $E\subset \mathbb{R}^{n+1}$ and a fixed $0<\alpha<1$, define its \textit{$\text{Lip}_\alpha$ $1/2$-caloric capacity} as
		\begin{equation*}
			\gamma_{\Theta^{1/2},\text{Lip}_\alpha}(E)=\sup |\langle T, 1 \rangle| ,
		\end{equation*}
		where the supremum is taken among all distributions $T$ with $\text{supp}(T)\subseteq E$ and satisfying $\|P\ast T\|_{\text{Lip}_\alpha}\leq 1$. Such distributions will be called \textit{admissible for $\gamma_{\Theta^{1/2},\text{\normalfont{Lip}}_\alpha}(E)$}.
	\end{defn}
	
	\bigskip
	\subsection{Comparability of \mathinhead{\gamma_{\Theta^{1/2},\text{Lip}_\alpha}}{} to the Hausdorff measure}
	\label{subsec6.1}
	\bigskip
	
	In the current setting, distributions admissible for the $\text{Lip}_\alpha$ $1/2$-caloric capacity exhibit a different growth condition to that of the genuine $1/2$-caloric capacity and its BMO variant.
	
	\begin{thm}
		\label{thm6.1}
		Let $T$ be a distribution in $\mathbb{R}^{n+1}$ with $\|P\ast T\|_{\text{\normalfont{Lip}}_\alpha}\leq 1$. If $\varphi$ is a $\pazocal{C}^1$ function supported on $Q\subset \mathbb{R}^{n+1}$ with $\|\nabla \varphi\|_\infty \leq \ell(Q)^{-1}$, then 
		\begin{equation*}
			|\langle T, \varphi \rangle |\lesssim \ell(Q)^{n+\alpha}.
		\end{equation*}
	\end{thm}
	
	\begin{proof}
		Let $T$ and $\varphi$ satisfy the conditions of the statement. Write $Q=Q(\overline{x}_0, \ell(Q))$ and proceed as in the proof of Theorem \ref{thm5.1},
		\begin{align*}
			|\langle T, \varphi\rangle| \leq |\langle P\ast T - P\ast T(\overline{x}_0), (-\Delta)^{-1/2}\varphi\rangle|+|\langle P\ast T - P\ast T(\overline{x}_0), \partial_t\varphi \rangle|=:I_1+I_2.
		\end{align*}
		Concerning $I_2$, the normalization conditions $\|\partial_t\varphi\|_\infty\leq \|\nabla\varphi\|_\infty\leq \ell(Q)^{-1}$ and $\|P\ast T\|_{\text{\normalfont{Lip}}_\alpha}\leq 1$ now imply
		\begin{align*}
			I_2&\leq \int_{Q}\big\rvert P\ast T(\overline{x})-P\ast T(\overline{x}_0)\big\rvert \big\rvert\partial_t\varphi(\overline{x})\big\rvert \text{d}\pazocal{L}^{n+1}(\overline{x}) \leq \frac{1}{\ell(Q)} \int_Q |\overline{x}-\overline{x}_0|^{\alpha} \text{d}\pazocal{L}^{n+1}(\overline{x}) \leq \ell(Q)^{n+\alpha}.
		\end{align*}
		To deal with the remaining integral $I_1$, proceed as in Theorem \ref{thm5.1} to deduce for the corresponding term $I_{11}$,
		\begin{align*}
			I_{11}&\leq \bigg( \int_{2Q} \big\rvert P\ast T(\overline{x})-P\ast T(\overline{x}_0)\big\rvert^2 \text{d}\pazocal{L}^{n+1}(\overline{x}) \bigg)^{1/2}\bigg( \int_{2Q} \big\rvert (-\Delta)^{1/2}\varphi(\overline{x})\big\rvert^2 \text{d}\pazocal{L}^{n+1}(\overline{x}) \bigg)^{1/2}\\
			&\lesssim \ell(Q)^{\alpha+(n+1)/2} \Bigg( \sum_{j=1}^n\|\partial_j\varphi\|_{L^2(2Q)}\Bigg) \lesssim \ell(Q)^{\alpha+(n+1)/2} \bigg(\frac{1}{\ell(Q)}\ell(Q)^{(n+1)/2}\bigg)= \ell(Q)^{n+\alpha}.
		\end{align*}
		Regarding the current term $I_{12}$, we also name $f:=P\ast T$ so that now
		\begin{align*}
			I_{12}&\lesssim \frac{1}{\ell(Q)} \sum_{j=1}^\infty \frac{1}{2^{j(n+1)}} \int_{C_{j+1}\setminus{C_j}} \big\rvert f(\overline{x})-f(\overline{x}_0) \big\rvert \text{d}\pazocal{L}^{n+1}\\
			&\leq \frac{1}{\ell(Q)} \sum_{j=1}^\infty \frac{1}{2^{j(n+1)}} \bigg( \int_{2^jQ} | \overline{x}-\overline{x}_0|^{\alpha q} d\pazocal{L}^{n+1}(\overline{x}) \bigg)^{1/q}\pazocal{L}^{n+1}(C_{j+1}\setminus{C_j})^{1/q'}\\
			&\lesssim \frac{1}{\ell(Q)} \sum_{j=1}^\infty \frac{1}{2^{j(n+1)}} \big( 2^j\ell(Q) \big)^{\alpha}\big( 2^j\ell(Q)\big)^{(n+1)/q}\ell(Q)^{(n+1)/q'}\,2^{(jn+1)/q'} \\
			&= \ell(Q)^{n+\alpha}\sum_{j=1}^\infty \frac{2^{1/q'}}{2^{j(1-\alpha-1/q)}},
		\end{align*}
		that converges if and only if $q>1/(1-\alpha)$. So fixing a proper $q$ we deduce the desired estimate.
	\end{proof}
	
	The previous growth condition implies an analogous result to Theorem \ref{thm5.2}, that reads as follows:
	\begin{thm}
		\label{thm6.2}
		Let $E\subset \mathbb{R}^{n+1}$ be a compact subset with $\pazocal{H}^{n+\alpha}(E)<\infty$ and $T$ an admissible distribution for $\gamma_{\Theta^{1/2},\text{\normalfont{Lip}}_\alpha}(E)$. Then $T$ is a signed measure which is absolutely continuous with respect to $\pazocal{H}^{n+\alpha}|_E$ and there exists a Borel function $f:E\to \mathbb{R}$ such that $T=f\,\pazocal{H}^{n+\alpha}|_E$ and  $\|f\|_{L^\infty(\pazocal{H}^{n+\alpha}|_E)}\lesssim 1$.
	\end{thm}
	
	The above statement follows from Theorem \ref{thm5.1} and the following lemma.
	\begin{lem}
		\label{lem6.3}
		Let $E\subset \mathbb{R}^{n+1}$ be a compact subset with $\pazocal{H}^{n+\alpha}(E)<\infty$ and $T$ a distribution supported on $E$ with $(n+\alpha)$-growth. Then $T$ is a signed measure which is absolutely continuous with respect to $\pazocal{H}^{n+\alpha}|_E$ and there exists a Borel function $f:E\to \mathbb{R}$ such that $T=f\,\pazocal{H}^{n+\alpha}|_E$ and  $\|f\|_{L^\infty(\pazocal{H}^{n+\alpha}|_E)}\lesssim 1$. 
	\end{lem}
	
	\begin{proof}
		The result follows by the same arguments given in \cite[Lemma 5.2]{MPr}, just by changing the rate of growth from $n$ to $n+\alpha$ and using Theorem \ref{thm6.1} instead of \cite[Corollary 3.3]{MPr}.
	\end{proof}
	
	Finally, let us present a similar result to that of Theorem \ref{thm5.3} in the current $\text{Lip}_\alpha$ setting, which in turn is analogous to \cite[Theorem 1]{Me}.
	
	\begin{thm}
		\label{thm6.4}
		Let $0<\alpha<1$. Then, there exist $($dimensional$\,)$ constants $C_1,C_2>0$ so that for any compact subset $E\subset \mathbb{R}^{n+1}$
		\begin{equation*}
			C_1 \pazocal{H}^{n+\alpha}_{\infty}(E)\leq \gamma_{\Theta^{1/2}, \text{\normalfont{Lip}}_\alpha}(E)\leq C_2\pazocal{H}_{\infty}^{n+\alpha}(E).
		\end{equation*}
	\end{thm}
	\begin{proof}
		For the right-hand side inequality consider $\varepsilon>0$ and $\{A_k\}_k$ a collection of sets in $\mathbb{R}^{n+1}$ that cover $E$ such that
		\begin{equation*}
			\sum_{k=1}^\infty \text{diam}(A_k)^{n+\alpha}\leq \pazocal{H}^{n+\alpha}_{\infty}(E)+\varepsilon.
		\end{equation*}
		By Theorem \ref{thm6.1}, the same argument given for Theorem \ref{thm5.3} yields the estimate.\newpage
		For the left-hand side inequality we will also apply Frostman's lemma. Assume $\pazocal{H}_{\infty}^{n+\alpha}(E)>0$ and consider a non trivial positive Borel regular measure $\mu$ supported on $E$ with $\mu(E)\geq c\pazocal{H}_{\infty}^{n+\alpha}(E)$ and $\mu(B(\overline{x},r))\leq r^{n+\alpha}$ for all $\overline{x}\in\mathbb{R}^{n+1}$, $r>0$. If we prove that
		\begin{equation*}
			\|P\ast \mu\|_{\text{Lip}_\alpha}\lesssim 1,
		\end{equation*}
		we will be done. Choose $\overline{x},\overline{y}\in\mathbb{R}^{n+1},\,\overline{x}\neq \overline{y},$ and consider the following partition of $\mathbb{R}^{n+1}$,
		\begin{align*}
			R_1:&= \big\{\overline{z} \;:\;|\overline{x}-\overline{y}|\leq |\overline{x}-\overline{z}|/2\big\}\cup\big\{\overline{z} \;:\;|\overline{y}-\overline{x}|\leq |\overline{y}-\overline{z}|/2\big\},\\
			R_2:= \mathbb{R}^{n+1}\setminus{R_1}&=\big\{\overline{z} \;:\;|\overline{x}-\overline{y}|> |\overline{x}-\overline{z}|/2\big\}\cap\big\{\overline{z} \;:\;|\overline{y}-\overline{x}|> |\overline{y}-\overline{z}|/2\big\},
		\end{align*}
		with their corresponding characteristic functions $\chi_1,\chi_2$ respectively.
		This way, we have
		\begin{align*}
			\frac{|P\ast \mu(\overline{x})-P\ast \mu(\overline{y})|}{|\overline{x}-\overline{y}|^\alpha} \leq \frac{1}{|\overline{x}-\overline{y}|^\alpha}&\int_{ |\overline{x}-\overline{y}|\leq |\overline{x}-\overline{z}|/2}|P(\overline{x}-\overline{z})-P(\overline{y}-\overline{z})|\text{d}\mu(\overline{z})\\
			&\hspace{-0.75cm}+\frac{1}{|\overline{x}-\overline{y}|^\alpha}\int_{ |\overline{y}-\overline{x}|\leq |\overline{y}-\overline{z}|/2}|P(\overline{x}-\overline{z})-P(\overline{y}-\overline{z})|\text{d}\mu(\overline{z})\\
			&\hspace{-0.75cm}+\frac{1}{|\overline{x}-\overline{y}|^\alpha}\int_{R_2}|P(\overline{x}-\overline{z})-P(\overline{y}-\overline{z})|\text{d}\mu(\overline{z})=:I_1+I_2+I_3.
		\end{align*}
		Regarding $I_1$, the third estimate of \cite[Lemma 2.1]{MPr} yields
		\begin{align*}
			I_1&\lesssim\frac{1}{|\overline{x}-\overline{y}|^\alpha}\int_{ |\overline{x}-\overline{y}|\leq |\overline{x}-\overline{z}|/2 } \frac{|\overline{x}-\overline{y}|}{|\overline{x}-\overline{z}|^{n+1}}\text{d}\mu(\overline{z}).
		\end{align*}
		Split the previous domain of integration into the annuli 
		\begin{equation*}
			A_j:=2^{j+1}B\big(\overline{x}, |\overline{x}-\overline{y}|\big)\setminus{2^{j}B\big(\overline{x}, |\overline{x}-\overline{y}|\big)}, \hspace{0.5cm} \text{for} \hspace{0.5cm} j\geq 1,    
		\end{equation*}
		and use that $\mu$ has growth of degree $n+\alpha$ with constant $1$ to deduce
		\begin{align*}
			I_1&\lesssim \frac{1}{|\overline{x}-\overline{y}|^{\alpha-1}}\sum_{j=1}^\infty\int_{A_j} \frac{\text{d}\mu(\overline{z})}{|\overline{x}-\overline{z}|^{n+1}}\lesssim  \frac{1}{|\overline{x}-\overline{y}|^{\alpha-1}}\sum_{j=1}^\infty \frac{(2^{j+1}|\overline{x}-\overline{y}|)^{n+\alpha}}{(2^{j}|\overline{x}-\overline{y}|)^{n+1}} \lesssim \sum_{j=1}^\infty\frac{1}{2^{(1-\alpha)j}}\lesssim 1,
		\end{align*}
		that is what we wanted to see. For $I_2$, we argue as in $I_1$ just interchanging the roles of $\overline{x}$ and $\overline{y}$. Finally, for $I_3$, observe that
		\begin{align*}
			I_3&\leq \frac{1}{|\overline{x}-\overline{y}|^\alpha}\int_{R_2}\frac{\text{d}\mu(\overline{z})}{|\overline{x}-\overline{z}|^{n}}+\frac{1}{|\overline{x}-\overline{y}|^\alpha}\int_{R_2}\frac{\text{d}\mu(\overline{z})}{|\overline{y}-\overline{z}|^{n}}\\
			&\leq \frac{1}{|\overline{x}-\overline{y}|^\alpha}\int_{|\overline{x}-\overline{y}|>|\overline{x}-\overline{z}|/2}\frac{\text{d}\mu(\overline{z})}{|\overline{x}-\overline{z}|^{n}}+\frac{1}{|\overline{x}-\overline{y}|^\alpha}\int_{|\overline{y}-\overline{x}|>|\overline{y}-\overline{z}|/2}\frac{\text{d}\mu(\overline{z})}{|\overline{y}-\overline{z}|^{n}} =:I_{31}+I_{32}.
		\end{align*}
		Concerning $I_{31}$, split the domain of integration into the (decreasing) annuli
		\begin{equation*}
			\widetilde{A}_j:=2^{-j}B\big(\overline{x}, |\overline{x}-\overline{y}|\big)\setminus{2^{-j-1}B\big(\overline{x}, |\overline{x}-\overline{y}|\big)}, \hspace{0.5cm} \text{for} \hspace{0.5cm} j\geq -1.    
		\end{equation*}
		Thus, in this case we have
		\begin{align*}
			I_{31}&\lesssim \frac{1}{|\overline{x}-\overline{y}|^{\alpha}}\sum_{j=-1}^\infty\int_{\widetilde{A}_j} \frac{\text{d}\mu(\overline{z})}{|\overline{x}-\overline{z}|^{n}}\lesssim  \frac{1}{|\overline{x}-\overline{y}|_p^{\alpha}}\sum_{j=-1}^\infty \frac{(2^{-j}|\overline{x}-\overline{y}|)^{n+\alpha}}{(2^{-j-1}|\overline{x}-\overline{y}|)^{n}} \lesssim \sum_{j=-1}^\infty\frac{1}{2^{\alpha j}}\lesssim 1.
		\end{align*}
		To obtain $I_{32}\lesssim 1$ we argue as for $I_{31}$ and interchanging the roles of $\overline{x}$ and $\overline{y}$. Combining the estimates obtained for $I_1, I_2$ and $I_3$ we deduce 
		\begin{equation*}
			\frac{|P\ast \mu(\overline{x})-P\ast \mu(\overline{y})|}{|\overline{x}-\overline{y}|^{\alpha}}\lesssim 1,
		\end{equation*}
		and since the points $\overline{x}\neq \overline{y}$ were arbitrarily chosen, we deduce the $\text{Lip}_\alpha$ condition.
	\end{proof}

	\vspace{1.5cm}
	{\small
		\begin{tabular}{@{}l}
			\textsc{Joan\ Hernández,} \\ \textsc{Departament de Matem\`{a}tiques, Universitat Aut\`{o}noma de Barcelona,}\\
			\textsc{08193, Bellaterra (Barcelona), Catalonia.}\\
			{\it E-mail address}\,: \href{mailto:joan.hernandez@uab.cat}{\tt{joan.hernandez@uab.cat}}
		\end{tabular}
	}
\end{document}